\documentclass[11pt]{amsart}

\usepackage{graphicx}%
\usepackage{multirow}%
\usepackage{amsmath,amssymb,amsfonts}%
\usepackage{amsthm}%
\usepackage{mathrsfs}%
\usepackage[title]{appendix}%
\usepackage{xcolor}%
\usepackage{textcomp}%
\usepackage{manyfoot}%
\usepackage{booktabs}%
\usepackage{algorithm}%
\usepackage{algorithmicx}%
\usepackage{algpseudocode}%
\usepackage{listings}%

\usepackage{enumitem}

\newtheorem{theorem}{Theorem}
\newtheorem{remark}[theorem]{Remark}

\begin{document}

\title[Coherent Structures in Flame Fronts]{Coherent Structures in Flame Fronts}

\author{Sultan Aitzhan}
\address{Department of Mathematics, Drexel University, Philadelphia, PA 19104 USA}
\email{sa3697@drexel.edu}

\author{Benjamin F. Akers}
\address{Department of Mathematics and Statistics, Air Force Institute of Technology, WPAFB, OH
45433, USA}
\email{benjamin.akers@afit.edu}

\author{David M. Ambrose}
\address{Department of Mathematics, Drexel University, Philadelphia, PA 19104 USA}
\email{dma68@drexel.edu}

\begin{abstract}
We study traveling waves in a coordinate-free model of flame fronts.  The flame front is the interface between the burnt and
unburnt phases of a gas undergoing combustion.  The front therefore moves in a preferred direction, as the unburnt gas is
consumed.  In the horizontally periodic, vertically unbounded setting, we prove the existence of waves of permanent
form which are traveling in the vertical direction.  We also compute these waves.  The analysis and computation use 
the framework of traveling waves in the arclength parameterization as previously developed by two of the authors and Wright.
\end{abstract}



\maketitle

\section{Introduction}

We study a coordinate-free model of flame fronts introduced by Frankel and Sivashinsky \cite{frankelSivashinsky1987}.  This model
specifies the velocity of the one-dimensional interface between the burnt and unburnt phases of a two-dimensional gas undergoing combustion; this interface is the flame
front.  This coordinate-free model specifies only the normal velocity of the front, in terms of intrinsic geometric information (namely, 
the arclength and curvature of the front).  Naturally, as the gas undergoes combustion, the unburnt phase of the gas is consumed, 
and thus the burnt phase can be viewed as advancing into the burnt phase.  This gives the motion of the front a naturally preferred direction,
and this is reflected in the coordinate-free model.  

Frankel and Sivashinsky give both a linear and a nonlinear version of the coordinate-free model; the linear model is linear with respect to
curvature, and is
\begin{equation}\label{LinModel}
 -U=1+(\alpha-1)\kappa+4\kappa_{ss}.
 \end{equation}
 Here, $U$ is the normal velocity of the front, $\kappa$ is the curvature of the front, $s$ is arclength measured from a specific point,
 and $\alpha$ is a parameter that reflects how unstable the front is.
The nonlinear version of the model (again, meaning nonlinear with respect to curvature) is
\begin{equation}\label{NonLinModel}
 -U=1+(\alpha-1)\kappa+\alpha^2(\alpha+3)\kappa_{ss}+\left(1+\frac{\alpha}{2}\right)\kappa^2+\left( 2\alpha+5\alpha^2-\frac{\alpha^3}{3} \right)\kappa^3. 
 \end{equation}

Frankel and Sivashinksy also introduced a coordinate-free model for the motion of two-dimensional flame fronts (between the phases of
a three-dimensional gas) \cite{frankelSivashinsky1988}.  In both the one-dimensional and two-dimensional cases, in the papers
\cite{frankelSivashinsky1987} and \cite{frankelSivashinsky1988}  the Kuramoto-Sivashinsky equation was formally derived as a weakly nonlinear model approximating the coordinate-free models.  Other prior work on coordinate-free models includes the the effects of temperature \cite{temperature1}, and the study of quasi-steady models \cite{roytburd1}, \cite{quasiSteady2}.  For the model of
\cite{temperature1} with temperature effects, an additional weakly nonlinear model other than the Kuramoto-Sivashinsky equation was 
studied in \cite{braunerEtAl}.  Some of these works included numerical simulations by fully explicit finite-difference methods.

More recently, the second and third authors introduced a non-stiff numerical method for the treatment of coordinate-free models.  The fully explicit
numerical method of, for instance, \cite{temperature1} faces a fourth-order stiffness constraint.  Instead, by following the framework of
Hou, Lowengrub, and Shelley for the evolution of vortex sheets with surface tension \cite{HLS1}, \cite{HLS2}, an implicit-explicit method
with at most a first-order stiffness constraint was developed \cite{ANZIAM}.  The third author and Hadadifard and Wright also made rigorous the
asymptotic derivation of the one-dimensional Kuramoto-Sivashinsky equation from the coordinate-free model \cite{SIADS}; this work included
expressing the coordinate-free model in evolutionary form, and proving well-posedness of the initial value problem for the coordinate-free model
of \cite{frankelSivashinsky1987}.  The third author and Liu then demonstrated the well-posedness of the higher-dimensional coordinate-free model
of \cite{frankelSivashinsky1988} in the work \cite{PHYSD}.

The second and third authors and Wright have developed and applied a framework for traveling waves in 
interfacial fluid dynamics which uses a parameterized curve, first in  \cite{IFB2}, and subsequently in \cite{WW}, \cite{EJMB-F}, \cite{ZAMP}, \cite{EJAM2},
\cite{PRSA}.  This is used in place of the usual traveling wave ansatz, that quantities be
a function of $x-ct,$ with spatial variable $x,$ wave speed $c,$ and time, $t.$   The alternative traveling wave ansatz is simply that the velocity of the curve 
satisfies
\begin{equation}\label{firstTravelingWaveAnsatz}
(x,y)_{t}=(c,0).
\end{equation}
The curve also satisfies
\begin{equation}\label{genericEvolution}
(x,y)_{t}=U\mathbf{\hat{n}}+V\mathbf{\hat{t}},
\end{equation}
where $U$ is the normal velocity (derived from physical principles such as the Euler equations), $V$ is the tangential velocity,
and $\mathbf{\hat{t}}$ and $\mathbf{\hat{n}}$ are the usual frame of unit tangent and normal vectors.
Combining \eqref{firstTravelingWaveAnsatz} and \eqref{genericEvolution}, one then can solve for the velocity of a traveling wave;
one finds $U=-c\sin(\theta),$ where $\theta$ is the tangent angle the curve forms with the horizontal.
If needed, additional equations (if the system has additional unknowns) can be found by differentiating with respect to time.
In the present case of advancing flame fronts, we will replace \eqref{firstTravelingWaveAnsatz} with a vertically-traveling ansatz for the front,
but will otherwise use the framework of \cite{IFB2}.

The Kuramoto-Sivashinsky equation in general dimension is
\begin{equation}\label{kuramotoSivashinsky}
\phi_{t}+\frac{1}{2}|\nabla\phi|^{2}=-\Delta^{2}\phi-\Delta\phi,
\end{equation}
and the position of the flame front would be $(x,\phi(x,t)).$
The Kuramoto-Sivashinsky equation  has long been known to exhibit rich dynamics and a number
of coherent structures, such as traveling waves \cite{troy2}, \cite{troy1}.  For the one-dimensional Kuramoto-Sivashinsky equation, 
the natural bifurcation parameter is the length of the periodic domain $[0,L];$ when $L$ passes through even multiples of $2\pi,$ a branch of
nontrivial traveling waves bifurcates from the zero solution.  The initial value problem for the Kuramoto-Sivashinsky equation is always well-posed,
and thus it is possible for these traveling waves to sometimes be asymptotically stable in the evolution.  This is what occurs for certain of the coherrent
structures.  We mention the detailed discussion of long-time behaviors of the two-dimensional Kuramoto-Sivashinsky equation \cite{papageorgiouEtAl},
which indeed shows the asymptotic stability of coherent structures for certain parameter values.

For the linear model \eqref{LinModel}, the highest-order derivative term has the correct sign to make the corresponding initial value problem well-posed.
We therefore can, and will below, consider stability of the traveling waves in the evolutionary problem.
By contrast, for the coherent structures we prove to exist for the nonlinear model \eqref{NonLinModel}, 
there will be no possibility of them being asymptotically stable in the nonlinear evolution.
Well-posedness of the initial value problem for the one-dimensional coordinate-free model with normal velocity \eqref{NonLinModel}
has been addressed in \cite{SIADS}.  
However, this well-posedness is not unconditional.
As can be seen in \eqref{kuramotoSivashinsky}, the fourth-order derivative has the correct sign for the Kuramoto-Sivashinksy equation to be a forward
parabolic equation.  However, as can be seen in \eqref{NonLinModel}, the sign of the fourth-derivative term in the coordinate-free model depends on the
value of the parameter $\alpha.$  In particular, for this coordinate-free model, the initial value problem will be well-posed for $\alpha>-3.$  The 
traveling waves we prove to exist, however, will be for $\alpha<-3.$  We see that the traveling waves for this coordinate-free model exist outside
the parameter range in which the Kuramoto-Sivashinsky equation was derived as a weakly nonlinear model, which was $\alpha\approx 1.$

The computations of branches of waves presented here use a pseudo-spectral collocation method, paired with a quasi-Newton iteration and numerical continuation.  Such procedures have been used in many previous works in a multiple physical scenarios \cite{fornberg1994review}. Spectral continuation based quasi-Newton iteration has been employed in water waves \cite{WW}, \cite{akersSeiders}, in elasticity problems \cite{liu2017quasi}, in trajectory control \cite{wang2018full}, and in laser-fluid interactions \cite{lane2024steady}, \cite{lane2023numerical},  to name a few.  Crucial to this procedure is understanding the leading asymptotics of the problem, as this can be used to create an initial guess for the algorithm.  In the water wave context these asymptotics are historically found via a Stokes' expansion \cite{craig2002traveling}, \cite{haupt1988modeling}, \cite{schwartz1974computer}, or a Lyapunov-Schmidt reduction \cite{akers2021wilton}, \cite{Akers_2025}.  This approach is also used here. 

The organization of the paper is as follows.  In Section \ref{formulationSection}, we adapt the traveling wave ansatz of \cite{IFB2} to the present setting.
In Section \ref{existenceSection} we apply Crandall-Rabinowitz bifurcation theory to prove existence of vertically traveling waves for the coordinate-free models
given by both \eqref{LinModel} and \eqref{NonLinModel}.  We compute branches of waves and observe their stability properties in Section \ref{computationSection}.
\section{Traveling wave formulation and geometric considerations}\label{formulationSection}
As discussed in the introduction, we will be following the traveling wave formulation of \cite{IFB2},
adapted to the context of a traveling front in a gas undergoing combustion.
Our front has position $(x(\sigma,t),y(\sigma,t))$ at time $t,$ with spatial parameter $\sigma.$
The physical context is that the appropriate speed of the front is $(0,-\beta),$ for some $\beta\in\mathbb{R}.$  
Adapting \eqref{firstTravelingWaveAnsatz} and restating \eqref{genericEvolution}, this becomes
\begin{equation}\nonumber
 (x_t,y_t)=(0,-\beta)=U\mathbf{\hat{n}} +V\mathbf{\hat{t}}.
 \end{equation}
Solving this for $U$ and $V,$ we find
\begin{equation}\label{UTravelingEquation}
U=-\beta\cos(\theta),
\end{equation}
\begin{equation}\label{VTravelingEquation}
V=-\beta\sin(\theta).
\end{equation}

\begin{remark} A more general traveling wave ansatz would be $(x,y)_{t}=(c,-\beta),$ so that the traveling wave could translate horizontally as well
as vertically.  However, both analytically and computationally, we have not found the existence of any such waves with nontrivial $c.$  Therefore we focus
on the above case, in which $c=0.$
\end{remark}

While we have mentioned that the normal velocity, $U,$ comes from one of the models \eqref{LinModel} or \eqref{NonLinModel}, we have not yet discussed
the tangential velocity.  As in the numerical works of Hou, Lowengrub, and Shelley for the motion of vortex sheets with surface tension \cite{HLS1}, \cite{HLS2},
and in subsequent analytical works such as \cite{SIMA}, \cite{JMFM}, 
we use a tangential velocity which maintains a favorable parameterization of the curve $(x(\sigma,t),y(\sigma,t)).$  
We study the horizontally periodic case in which 
\begin{equation}\label{periodicity}
x(\sigma+2\pi,t)=x(\sigma,t)+2\pi,\qquad y(\sigma+2\pi,t)=y(\sigma,t),
\end{equation}
for all $\sigma$ and $t.$
With the arclength element $s_{\sigma}$ defined through the equation
\begin{equation}\nonumber
s_{\sigma}^{2}=x_{\sigma}^{2}+y_{\sigma}^{2},
\end{equation}
the length of one period of the curve is
\begin{equation}\nonumber
L(t)=\int_{0}^{2\pi}s_{\sigma}(\sigma,t)\ \mathrm{d}\sigma.
\end{equation}
In this setting, our preferred parameterization is a normalized arclength parameterization, so that
\begin{equation}\nonumber
s_{\sigma}(\sigma,t)=\frac{L(t)}{2\pi},
\end{equation}
for all $\sigma$ and $t.$  Thus in such a parameterization the arclength element is independent of $\sigma.$  In an evolutionary problem,
$s_{\sigma}$ would still depend on time, but in the present case it is also independent of time.  In the evolutionary problem, the evolution equation
for $s_{\sigma}$ is
\begin{equation}\nonumber
s_{\sigma t}=V_{\sigma}-\theta_{\sigma}U,
\end{equation}
which, when rearranged, can be taken to be the definition of the tangential velocity which maintains the normalized arclength parameterization:
\begin{equation}\nonumber
V_{\sigma}=\theta_{\sigma}U+\frac{L_{t}}{2\pi}.
\end{equation}
In the steady or traveling case, we have $L_{t}=0,$ so that $V_{\sigma}=\theta_{\sigma}U.$

In addition to $V_{\sigma}=\theta_{\sigma}U,$ however, we have another equation for $V,$ which is \eqref{VTravelingEquation}.
Taking the derivative of this with respect to $\sigma,$ we get
\begin{equation}\nonumber
V_{\sigma}=-\beta\theta_{\sigma}\cos(\theta(\sigma)).
\end{equation}
Substituting for $U$ from \eqref{UTravelingEquation}, we see that this is the same as saying $V_{\sigma}=\theta_{\sigma}U.$
Thus we view equation \eqref{UTravelingEquation} for the normal velocity $U$ as our traveling wave ansatz, and we view \eqref{VTravelingEquation} 
for the tangential velocity $V$ as setting the arclength parameterization of the curve.

Before closing the section, we make a further remark on the arclength and length of the curve.  We will be using $\theta$ as the unknown to be solved for
in both our analysis and our computations.  The length, $L,$ will appear in the traveling wave equations to be solved.  Therefore we must discuss how to 
find $L$ in terms of $\theta.$  This comes from the periodicity \eqref{periodicity}.  The relationship between the curve $(x,y)$ and the tangent angle $\theta,$
taking into account the normalized arclength parameterization, is
\begin{equation}\nonumber
(x_{\sigma},y_{\sigma})=\frac{L}{2\pi}(\cos(\theta(\sigma)),\sin(\theta(\sigma)).
\end{equation}
Integrating $x_{\sigma}$ over $[0,2\pi],$ we have
\begin{equation}\nonumber
2\pi=x(2\pi)-x(0)=\int_{0}^{2\pi}x_{\sigma}(\sigma)\ \mathrm{d}\sigma=\frac{L}{2\pi}\int_{0}^{2\pi}\cos(\sigma)\ \mathrm{d}\sigma.
\end{equation}
We therefore have our formula for $L$ in terms of $\theta,$
\begin{equation}\label{lengthFormula}
L=\frac{4\pi^{2}}{\displaystyle\int_{0}^{2\pi}\cos(\theta(\sigma))\ \mathrm{d}\sigma}.
\end{equation}

\section{Existence theorems for expanding traveling waves}\label{existenceSection}
To prove existence of nontrivial traveling solutions to the coordinate-free flame front model, 
our main tool is the Crandall-Rabinowitz ``bifurcation from a simple eigenvalue'' theorem, which we now state.
\begin{theorem}[\cite{Kielhofer2004}, Theorem I.5.1]\label{Crandall-Rabinowitz}
Suppose $X$ and $Z$ are Banach spaces, and $F: X \times \mathbb{R} \to Z$ is a mapping satisfying the properties below:
\begin{enumerate}[label=(\roman*)]
	\item\label{CRcond1} $F(0, \lambda) = 0$ for all $\lambda \in \mathbb{R},$
	\item\label{CRcond2} $F_u(0, \lambda_0)$ is a Fredholm operator of index zero for some $\lambda_0,$
	\item\label{CRcond3} $F \in C^2(U \times V, Z),$ where $0 \in U$ and $\lambda_0 \in V,$ for $U$ and $V$ open neighborhoods of $X$ and $\mathbb{R}$ respectively,
	\item\label{CRcond4} There exists $v_{0}\in X$ such that $\ker F_u(0, \lambda_0) = \mathrm{span}[v_0]$, satisfying $F_{u\lambda}(0, \lambda_0)v_0 \notin \mathrm{range}(F_u(0, \lambda_0)).$
\end{enumerate}
Then, there is a nontrivial continuously differentiable curve through $(0,\lambda_0),$
\[
	\{ (x(\varpi), \lambda(\varpi)): \varpi \in (-\delta, \delta), (x(0), \lambda(0)) = (0, \lambda_0) \},
\]
such that $F(x(\varpi), \lambda(\varpi)) = 0$ for $\varpi \in (-\delta, \delta).$
\end{theorem}
In the statement of this theorem, $F_u$ indicates the Fréchet derivative of $F$ with respect to $u.$

Before verifying the hypotheses of Theorem \ref{Crandall-Rabinowitz} for our problem, we must define the mapping $F$ and identify the relevant
function spaces.  Our mapping $F$ will be a version of \eqref{UTravelingEquation}, and so we now rewrite \eqref{UTravelingEquation} slightly, making use of 
\eqref{LinModel} and \eqref{NonLinModel}.  Specifically, we begin by working with \eqref{NonLinModel}.

Since \eqref{UTravelingEquation} is in terms of the tangent angle the front forms with the horizontal, $\theta,$ and since there is a natural relationship 
between $\theta$ and the curvature, $\kappa,$ we change variables in \eqref{NonLinModel} to $\theta.$  This relationship between $\theta$ and $\kappa$
is
\begin{equation}\nonumber
\kappa = \dfrac{\theta_\sigma}{s_\sigma}.
\end{equation}
Derivatives with respect to arclength satisfy
\begin{equation}\nonumber
\dfrac{d}{ds} = \dfrac{1}{s_\sigma} \dfrac{d}{d\sigma}.
\end{equation}
From these considerations, we find that the \eqref{NonLinModel} becomes 
\begin{equation}\label{NonLinModelRecast}
- U = 1 + (\alpha -1) \frac{\theta_\sigma}{s_\sigma} + \left( 1 + \frac{\alpha^2}{2} \right)  \frac{\theta_\sigma^2}{s_\sigma^2} + \left( 2\alpha + 5\alpha^2 - \frac{\alpha^3}{3} \right)  \frac{\theta_\sigma^3}{s_\sigma^3} + \alpha^2(\alpha +3) \frac{\theta_{\sigma\sigma\sigma}}{s_\sigma^3}.
\end{equation}
Combining \eqref{UTravelingEquation} and \eqref{NonLinModelRecast}, and letting $F$ be given by
\begin{multline}\nonumber
F(\theta, \beta;\alpha) = 1 + (\alpha -1) \frac{\theta_\sigma}{s_\sigma} + \left( 1 + \frac{\alpha^2}{2} \right)  \frac{\theta_\sigma^2}{s_\sigma^2} 
\\
+ \left( 2\alpha + 5\alpha^2 - \frac{\alpha^3}{3} \right)  \frac{\theta_\sigma^3}{s_\sigma^3} + \alpha^2(\alpha +3) \frac{\theta_{\sigma\sigma\sigma}}{s_\sigma^3} - \beta \cos(\theta),
\end{multline}
we see that solving for vertically traveling waves in the coordinate-free model is the same as solving $F(\theta, \beta; \alpha)=0$.
We take the space $X$ in the statement of Theorem \ref{Crandall-Rabinowitz} 
to be $H^{3}_{\text{odd}}\times\mathbb{R},$ 
and we take the space $Z$ there to be $L^{2}_{\text{even}}.$  The spaces $H^{3}_{\text{odd}}$ and $L^{2}_{\text{even}}$ 
are the usual periodic Sobolev and Lebesgue spaces $H^{3}$ and $L^{2},$ respectively,
restricted to functions with Fourier sine series expansions and Fourier cosine series expansions, respectively.
Applying Theorem \ref{Crandall-Rabinowitz} to this $F$ and with these function spaces, we have the following result.  

\begin{theorem}\label{firstExistenceTheorem}
For each integer $k_0 \geq 1,$ the polynomial
\begin{equation}\nonumber
q(\alpha)=(\alpha-1)-\alpha^{2}(\alpha+3)k_{0}^{2}
\end{equation}
has one and only one real root (with multiplicity one), $\alpha_{0}.$
The coordinate-free model with normal velocity \eqref{NonLinModel} 
has infinitely many traveling solutions satisfying \eqref{UTravelingEquation}, \eqref{VTravelingEquation} 
in the neighbourhood of $(\theta = 0, \beta = 1, \alpha = \alpha_0)$, 
i.e. there is a nontrivial continuously differentiable curve 
\[
	\{ (\theta(\varpi), \beta(\varpi), \alpha(\varpi)): \varpi \in (-\delta, \delta), ~\text{and}~ (\theta(0), \beta(0), \alpha(0)) = (0,1,\alpha_0) \},
\]
such that
$F(\theta(\varpi), \beta(\varpi), \alpha(\varpi)) = 0$ for $\varpi \in (-\delta, \delta)$.
\end{theorem}

\begin{proof}	
In order to apply Theorem \ref{Crandall-Rabinowitz}, we need to check the following items.
\begin{enumerate}[label=(\roman*)]
	\item\label{CFcond1} For some $\theta_0, \beta_0$, $F(\theta_0, \beta_0;\alpha) = 0$ for all $\alpha \in \mathbb{R}.$
	\item\label{CFcond2} Write $L(\alpha) := F_{(\theta, \beta)}(\theta = \theta_0, \beta = \beta_0; \alpha)$, so that $L: H^3_{\text{odd}} \times \mathbb{R} \to L^2_{\text{even}}$ is a linear operator.
	Then $L(\alpha_0)$ is Fredholm with Fredholm index $0$, for some $\alpha_0$.
	\item\label{CFcond3} $F \in C^2(H^3_{\text{odd}} \times \mathbb{R}\times\mathbb{R}, L^2_{\text{even}}).$ 
	\item\label{CFcond4} (Transversality) There exists $\phi_{0}\in H^3_{\text{odd}} \times \mathbb{R}$ such that $\ker L(\alpha_0)= \mathrm{span}[\phi_0]$, 
	satisfying $L_{\alpha}(\alpha_0)\phi_0 \notin \mathrm{range}(L(\alpha_0)).$
\end{enumerate} 
We note that the linear operator $L$ in this section is not the same as the length of the curve, also denoted $L,$ in the previous section.  We allow this notational 
collision because we believe that in context it will always be clear what is meant by $L.$
We proceed to verify the four required items.

To check \ref{CFcond1}, note that $F(\theta=0, \beta=1;\alpha) = 0$ for all $\alpha\in\mathbb{R},$ so we may take $(\theta_0, \beta_0) = (0,1).$

To check \ref{CFcond2}, we compute $L(\alpha_0)$ and verify that we may choose $\alpha_0$ so that $L(\alpha_0)$ is Fredholm with Fredholm index $0$. 
Linearizing $F$ about $\theta=0,$ $\beta=1,$ it is straightforward to compute that $L$ is given as follows:
\[ 
L(\alpha)(v, \gamma) = (\alpha -1 ) v_\sigma + \alpha^2(\alpha +3) v_{\sigma\sigma\sigma} - \gamma.
\]
The definition of the index of $L(\alpha)$ is 
\[
\mathrm{ind} L = \dim \ker L - \dim \mathrm{coker}~ L = \dim \ker L - \dim \left( L^2_{\text{even}}\setminus \mathrm{range}~(L)\right),
\]
so we must understand both this kernel and cokernel.
We first determine the kernel of $L(\alpha)$ by expanding $v$ in the Fourier sine series.
If we take the expansion
\begin{equation}\nonumber
v(\sigma)=\sum_{k\geq1}\hat{v}(k)\sin(k\sigma),
\end{equation}
then we may write $L(\alpha)(v,\gamma)$ as
\begin{equation}\nonumber
L(\alpha)(v,\gamma)=
\sum_{k \geq 1} k[(\alpha- 1) - \alpha^2(\alpha +3) k^2] \hat{v}(k)\cos(k\sigma) - \gamma.
\end{equation}
Setting this equal to zero, we immediately see we must have $\gamma=0.$  Then, for the element of the kernel to be nontrivial, for
some $k_{0}$ we must have
\begin{equation}\label{3rdDegPol}
	q(\alpha)=(\alpha- 1) - \alpha^2(\alpha +3) k_0^2 = 0.
\end{equation} 
Since the left-hand side is a cubic polynomial in $\alpha,$ it can have either one real root and two complex roots (that are a conjugate pair) or three real roots
(counting multiplicity).
Note that for a cubic polynomial of the form $ax^3 + bx^2 + cx + d$, the discriminant is given by \cite{Irving2004}
\[
\Delta =18abcd - 4 b^3 d + b^2 c^2 - 4ac^3 - 27 a^2 d^2.
\]
If $\Delta \geq 0$ then the polynomial has three real roots (in the case $\Delta=0$ this means real repeated roots), and if $\Delta < 0$, the polynomial has one real root and two complex roots.
In \eqref{3rdDegPol}, we have 
\[
a = - k_0^2, \quad b = -3k_0^2, \quad c = 1, \quad d = -1,
\]
so that upon simplifying we obtain
\[
\Delta = - 72k_0^4 - 108k_0^3 + 4k_0^2 = 4k_0^2(-18k_0^2 - 27k_0 + 1).
\]
Clearly, $-18k_0^2 - 27k_0 + 1<0$ for every $k_0 \geq 1,$ and so \eqref{3rdDegPol} must have a unique real solution, say $\alpha = \alpha_0$.
Furthermore, for any $k_{1}\neq k_{0},$ we have
\begin{equation}\nonumber
(\alpha_{0}-1)-\alpha_{0}^{2}(\alpha_{0}+3)k_{1}^{2}\neq0;
\end{equation}
we can see this because taking the derivative of $(\alpha_{0}-1)-\alpha_{0}^{2}(\alpha_{0}+3)k^{2}$ with respect to $k$ indicates that it is strictly monotone in $k$
for $k\geq1.$
From here, for a fixed integer $k_0 \geq 1,$ we see that $L(\alpha_0)(v=  \sin(k_0 \sigma), \gamma = 0) = 0.$
Therefore, we conclude that $\ker L(\alpha_0)= \mathrm{span}([\sin(k_0 \sigma), 0])$.
It follows that $\dim \ker L(\alpha_0) = 1$.

We now compute $\dim \left( L^2_{\text{even}}\setminus \text{range}(L(\alpha_{0}))\right).$ 
From the preceding paragraph, it is clear that 
\begin{equation}\nonumber
\text{range}(L(\alpha_0)) = \overline{\text{span}}\{ \cos(kx): k = 0, 1, \ldots ~ \mathrm{and}~k \neq k_0 \},
\end{equation}
where the closure is taken with respect to the norm in $L^{2}_{\text{even}}.$
Thus, $L^2_{\text{even}}\setminus \text{range}(L(\alpha_{0})) = \mathrm{span}~ [\cos(k_0x)].$ 
As such, 
\begin{equation}\nonumber
\dim \left(L^2_{\text{even}}\setminus \text{range}(L(\alpha_{0}))\right) = 1,
\end{equation} 
and so we find that 
$L(\alpha_0)$ is Fredholm with index $0$.

Since \ref{CFcond3} is clear, 
it only remains to show the second part of \ref{CFcond4}.  Specifically, we must show 
$L_{\alpha}(\alpha_0)(\sin(k_0 \sigma), 0) \notin \mathrm{range}(L(\alpha_0)).$
Now, since
\[
	L_{\alpha}(v, \gamma) = v_\sigma + 3\alpha (\alpha + 2) v_{\sigma\sigma\sigma},
\]
we have
\[ L_{\alpha}(\alpha_0)(\sin(k_0 \sigma), 0) = k_{0} (- 1 + 3 k_{0}^2 \alpha_0 (\alpha_0 + 2)) \cos(k_0x).
\]
We must determine whether the coefficient of $\cos(k_{0}x)$ here is nonzero.  We know $k_{0}$ is nonzero, so it is enough to consider
whether $-1+3k_{0}^{2}\alpha_{0}(\alpha_{0}+2)=0.$  We let $p(\alpha)$ be the quadratic polynomial
\begin{equation}\nonumber
p(\alpha)=-1+3k_{0}^{2}\alpha(\alpha+2)=3\alpha^{2}k_{0}^{2}+6\alpha k_{0}^{2}-1.
\end{equation}
We rewrite the cubic polynomial $q(\alpha)$ as
\begin{equation}\nonumber
q(\alpha)=-k_{0}^{2}\alpha^{3}-3k_{0}^{2}\alpha^{2}+\alpha-1.
\end{equation} 
Since $\alpha_{0}$ is the unique real root of $q,$ if we can conclude that $p$ and $q$ have no roots in common, then we have
that $p(\alpha_{0})\neq0.$  If the resultant of $p$ and $q$ is nonzero, then $p$ and $q$ have no roots in common \cite{gelfand}.  
The resultant of $p$ and $q$ is the determinant of their Sylvester matrix, namely
\begin{equation}\nonumber
\text{res}(p,q)=\left|
\begin{array}{ccccc}
-k_{0}^{2} & -3k_{0}^{2} & 1 & -1 & 0 \\
0 & -k_{0}^{2} & -3k_{0}^{2} & 1 & -1 \\
3k_{0}^{2} & 6k_{0}^{2} & -1 & 0 & 0 \\
0 & 3k_{0}^{2} & 6k_{0}^{2} & -1 & 0 \\
0 & 0 & 3k_{0}^{2} & 6k_{0}^{2} & -1 
\end{array}
\right|
=4k_{0}^{2}(27k_{0}^{4}+18k_{0}^{2}-1).
\end{equation}
The right-hand side is positive for $k_{0}\geq1,$ so we conclude $p(\alpha_{0})\neq0.$
Since we have shown that $L_{\alpha}(\alpha_0)(\sin(k_0 \sigma), 0)$ is a non-zero multiple of $\cos(k_0x)$, we see that 
\[
L_{\alpha}(\alpha_0)(\sin(k_0 \sigma), 0) \notin \mathrm{range}(L(\alpha_0)) = \{ \cos(kx): k = 0, 1, \ldots ~ \mathrm{and}~k \neq k_0 \}.
\]
This concludes the proof of transversality.

Having checked the required conditions, an application of the Crandall-Rabinowitz theorem (Theorem \ref{Crandall-Rabinowitz}) yields the desired result.
\end{proof}

We have the corresponding result on existence of traveling waves using the linear velocity \eqref{LinModel}.  In this case, we seek to solve $G(\theta,\beta;\alpha)=0,$
where
\begin{equation}\nonumber
G(\theta,\beta;\alpha)=1+(\alpha-1)\frac{\theta_{\sigma}}{s_{\sigma}}+\frac{4\theta_{\sigma\sigma\sigma}}{s_{\sigma}^{3}}-\beta\cos(\theta).
\end{equation}
The function spaces $X$ and $Z$ remain unchanged, namely $X=H^{3}_{\text{odd}}\times\mathbb{R},$ and $Z=L^{2}_{\mathrm{even}}.$

\begin{theorem}\label{secondExistenceTheorem}
For each integer $k_0 \geq 1,$ let $\alpha_{1}=4k_{0}^{2}+1.$  The coordinate-free model with normal velocity \eqref{LinModel} 
has infinitely many traveling solutions satisfying \eqref{UTravelingEquation}, \eqref{VTravelingEquation} 
in the neighbourhood of $(\theta = 0, \beta = 1, \alpha = \alpha_1)$, 
i.e. there is a nontrivial continuously differentiable curve 
\[
	\{ (\theta(\varpi), \beta(\varpi), \alpha(\varpi)): \varpi \in (-\delta, \delta), ~\text{and}~ (\theta(0), \beta(0), \alpha(0)) = (0,1,\alpha_1) \},
\]
such that
$G(\theta(\varpi), \beta(\varpi), \alpha(\varpi)) = 0$ for $\varpi \in (-\delta, \delta)$.
\end{theorem}

\begin{proof} We again have to verify the four items necessary for application of Theorem \ref{Crandall-Rabinowitz}.  We focus on the differences in the proof 
in the present situation and the proof of Theorem \ref{firstExistenceTheorem}.

In the present case, the linearization, $L(\alpha),$ is given by
\begin{equation}\nonumber
L(\alpha)(v,\gamma)=(\alpha-1)v_{\sigma}+4v_{\sigma\sigma\sigma}-\gamma
\\
=\sum_{k\geq1}k\left((\alpha-1)-4k^{2}\right)\hat{v}(k)\cos(k\sigma).
\end{equation}
This has one-dimensional kernel at $\alpha_{1}=4k_{0}^{2}+1$ for some $k_{0}\in\mathbb{N},$ $k_{0}\geq1.$
It is also clear, similarly to the proof of Theorem \ref{firstExistenceTheorem}, that $L(\alpha_{1})$ has one-dimensional cokernel 
(in fact, not just the dimension is the same, but the cokernel is the same as in the previous proof).  So, we have again concluded
that $L(\alpha_{1})$ is Fredholm of index zero.

We then need to check the transversality condition.  This requires checking that $L_{\alpha}(\alpha_{1})(\sin(k_{0}\sigma),0)\notin\text{range}(L(\alpha_{1}).$
But $L_{\alpha}(\alpha)(v,\gamma)=v_{\sigma},$ so 
\begin{equation}\nonumber
L_{\alpha}(\alpha_{1})(\sin(k_{0}\sigma),0)=k_{0}\cos(k_{0}\sigma).
\end{equation}  
This is in the cokernel, rather than
being in the range.  This completes the proof.
\end{proof}

\section{Computations of Traveling Waves}\label{computationSection}
In this section the coordinate-free models for flame fronts are numerically approximated with both choices of velocity closure 
(linear in curvature and nonlinear in curvature).  Traveling waves are computed in both models.  In the linear velocity closure, where the model is well posed in the parameter regime where traveling waves exist, the stability of the traveling waves are investigated by evolving perturbed waves in the time dynamic problem.  The numerical method for the evolution is that developed in \cite{ANZIAM}. 

Traveling waves are computed using a method similar to that in \cite{WW}, \cite{ZAMP}, where the solution is represented as a truncated Fourier series, 
\begin{equation}\label{FourierWave} \theta(\sigma)=\sum_{n=-Nx/2}^{Nx/2} a_ne^{in\sigma}.\end{equation}
As discussed above, both models support odd real solutions, so the $a_n$ are taken to be pure imaginary and odd in $n$.  A single wave profile \eqref{FourierWave} includes $Nx/2$ unknown Fourier coefficients.  The speed $\beta$ and the parameter $\alpha$ are also treated as unknowns.  The projection of the model equation into Fourier space gives $Nx/2$ equations (each Fourier mode of a real even function set to zero).  The system is closed with a specification of the wave amplitude. The resulting   nonlinear system of equations is solved via quasi-Newton iteration. The maximum of $\theta$ is chosen as a measure of amplitude.  Small amplitude waves are computed with an asymptotic initial guess, described below.  Larger amplitudes are computed with numerical continuation, similar to \cite{akers2010dynamics}, \cite{cho2014computation},
\cite{claassen2018numerical}, \cite{guan2022new}.

\begin{figure}[tp]
\centerline{\includegraphics[width=0.3\textwidth]{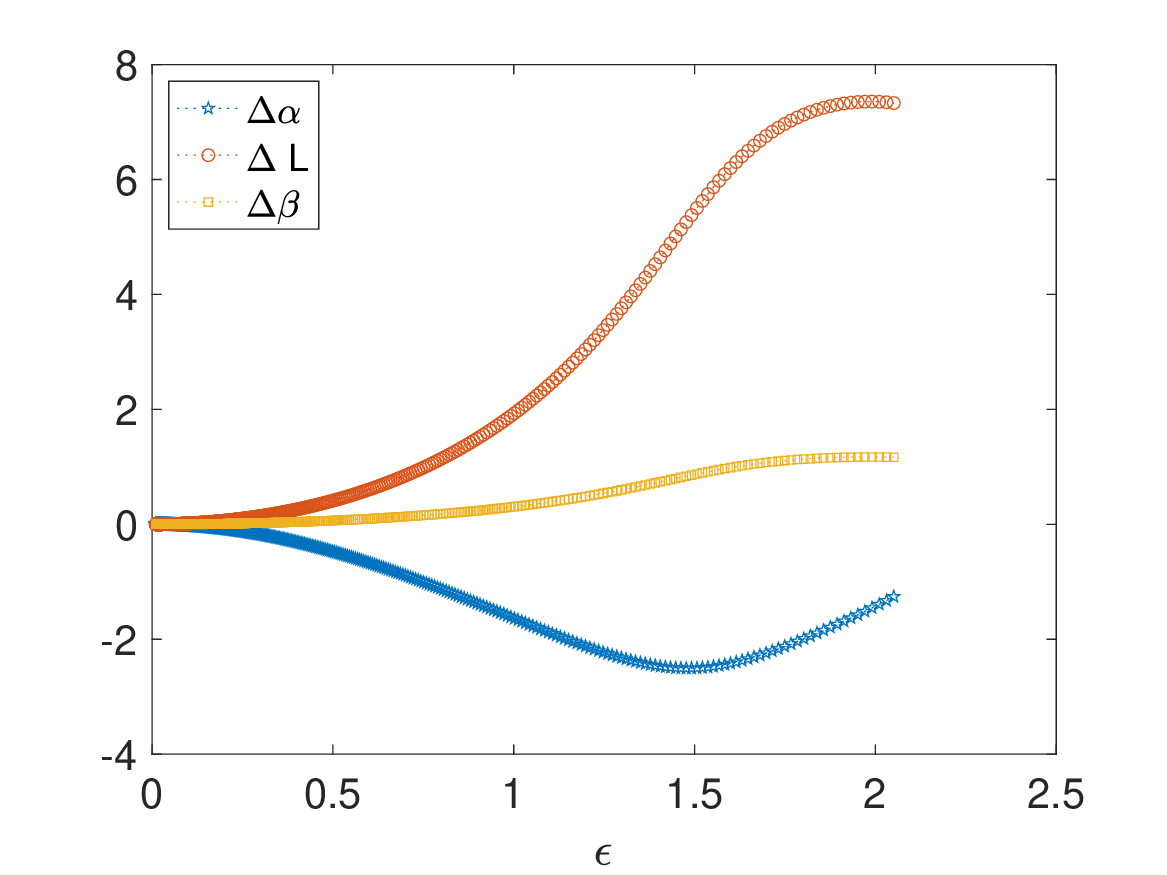}\includegraphics[width=0.3\textwidth]{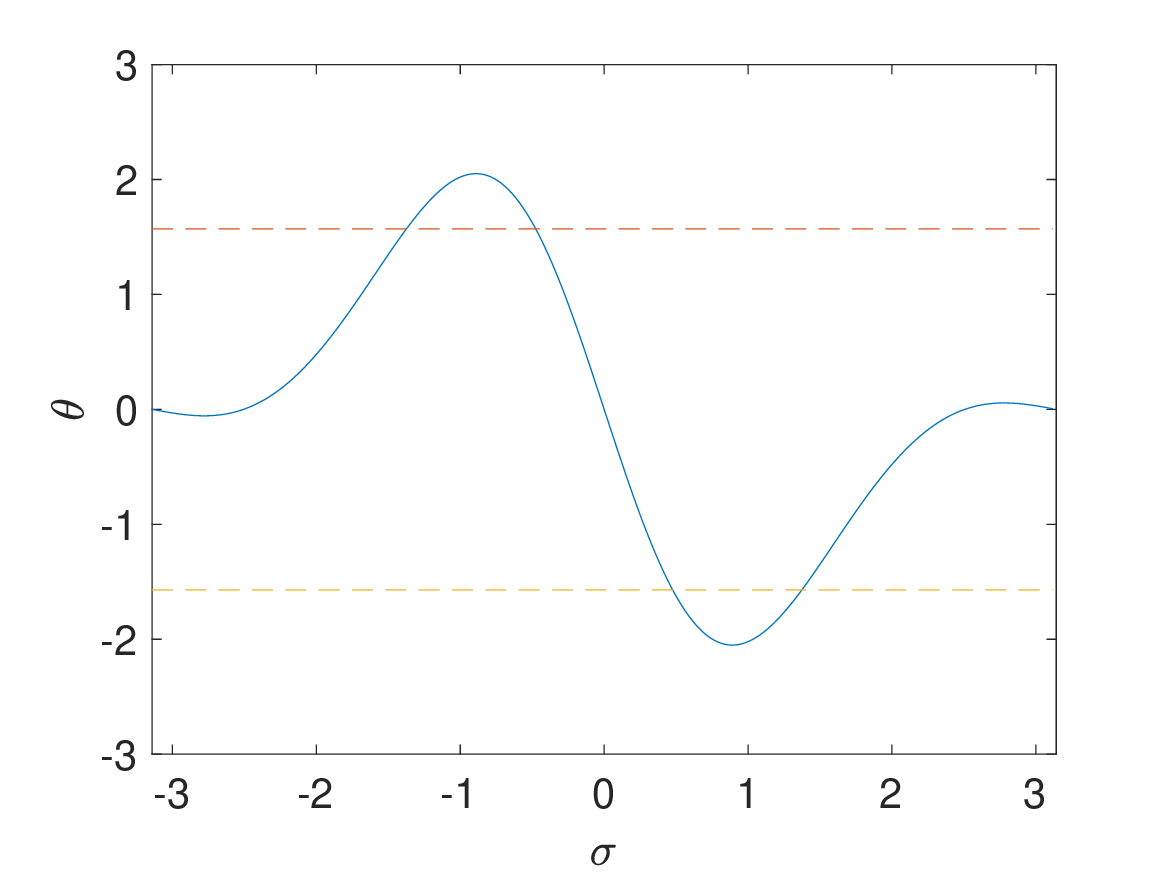}\includegraphics[width=0.3\textwidth]{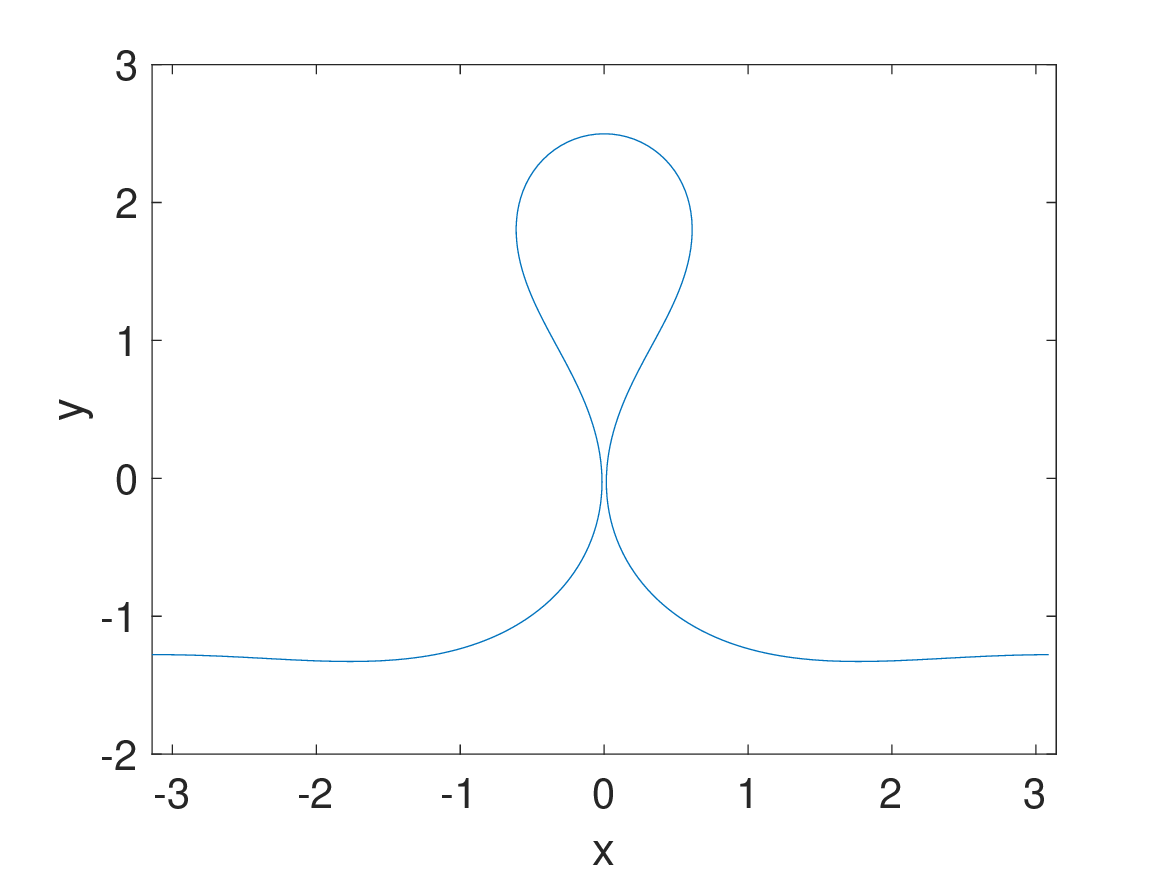}}
\caption{\it \textbf{Left:} The change in parameters $L,\alpha$ and $\beta$ over a computed branch of solutions to \eqref{StandingLinModel}, starting with $\alpha=5, \beta=1,L=2\pi$. \textbf{Center:} The standing wave's tangent angle $\theta$ on the largest wave on this part of the branch. The dotted lines mark $|\theta|=\pi/2$. \textbf{Right:} The standing wave's interface profile of the largest wave on this part of the branch.  The branch limits on a self-intersecting profile. \label{LinUBranch}  }
\end{figure}

\begin{figure}[tp]
\centerline{\includegraphics[width=0.3\textwidth]{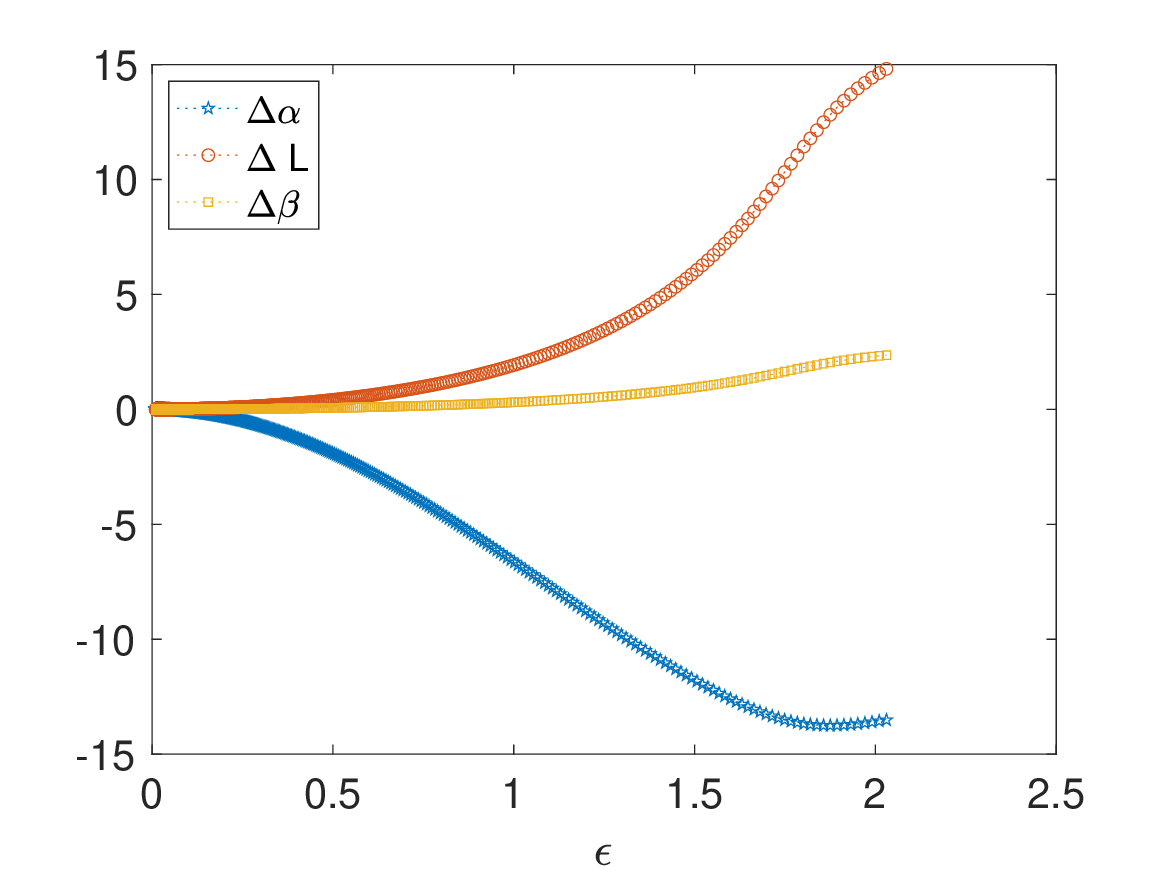}\includegraphics[width=0.3\textwidth]{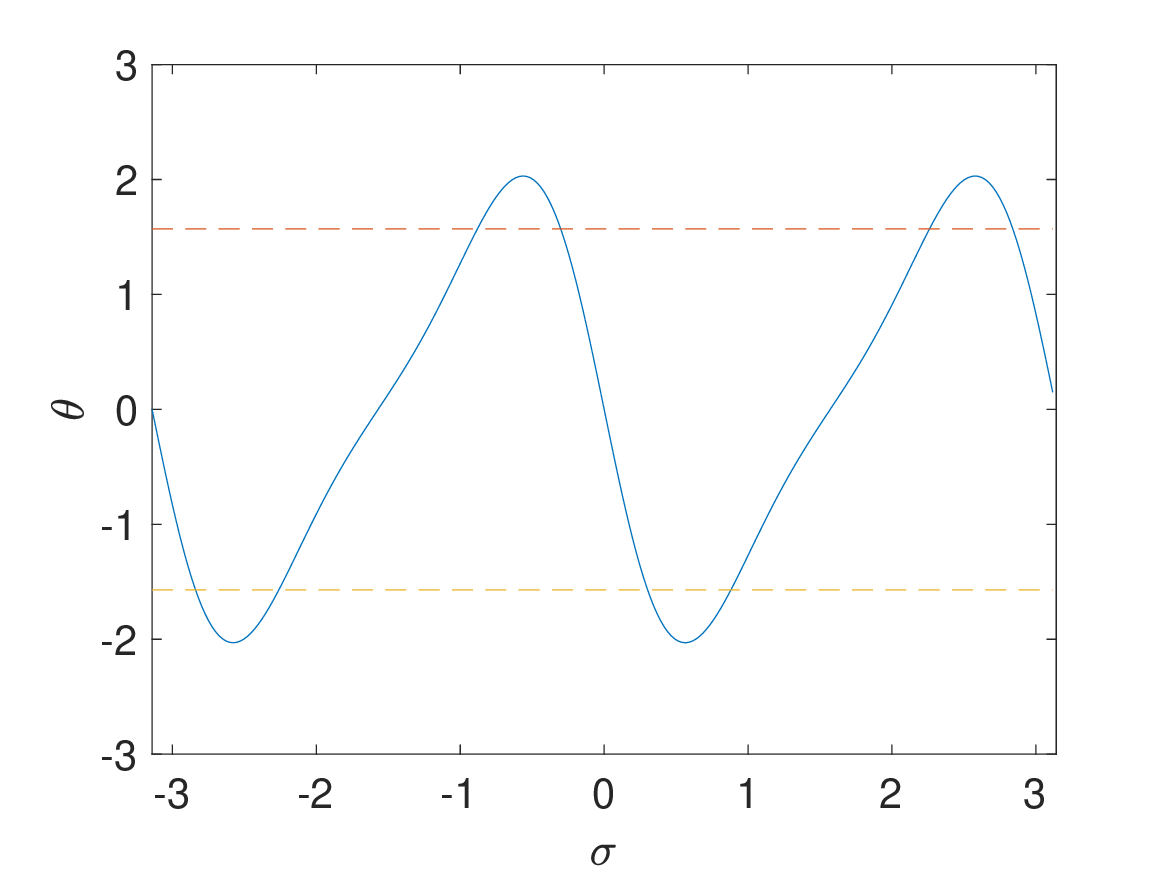}\includegraphics[width=0.3\textwidth]{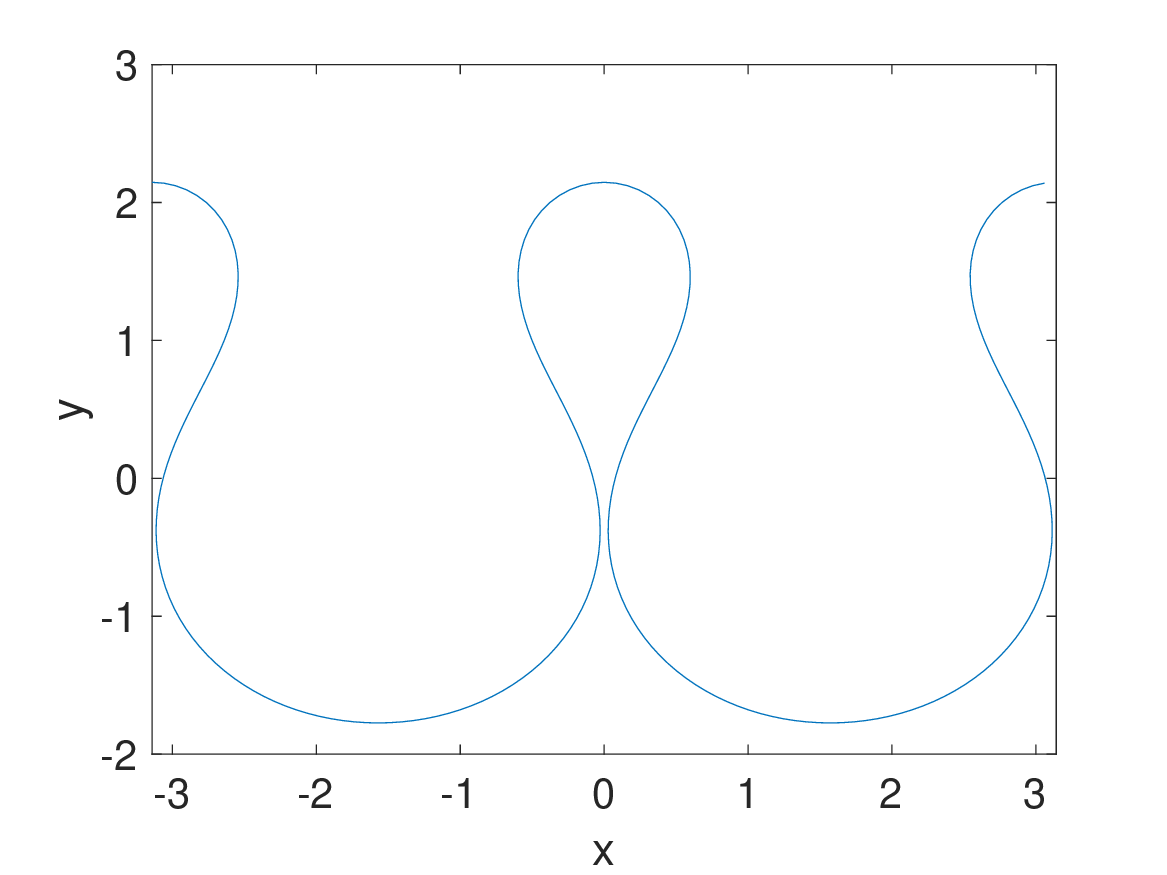}}
\caption{\it \textbf{Left:} The change in parameters $L,\alpha$ and $\beta$ over a computed branch of solutions to \eqref{StandingLinModel}, starting with $\alpha=17, \beta=1,L=2\pi$. \textbf{Center:} The standing wave's tangent angle $\theta$ on the largest wave on this part of the branch. The dotted lines mark $|\theta|=\pi/2$. \textbf{Right:} The standing wave's interface profile of the largest wave on this part of the branch.  The branch limits on a self intersecting profile. \label{LinUBranch2}  }
\end{figure}

\begin{figure}[tp]
\centerline{\includegraphics[width=0.3\textwidth]{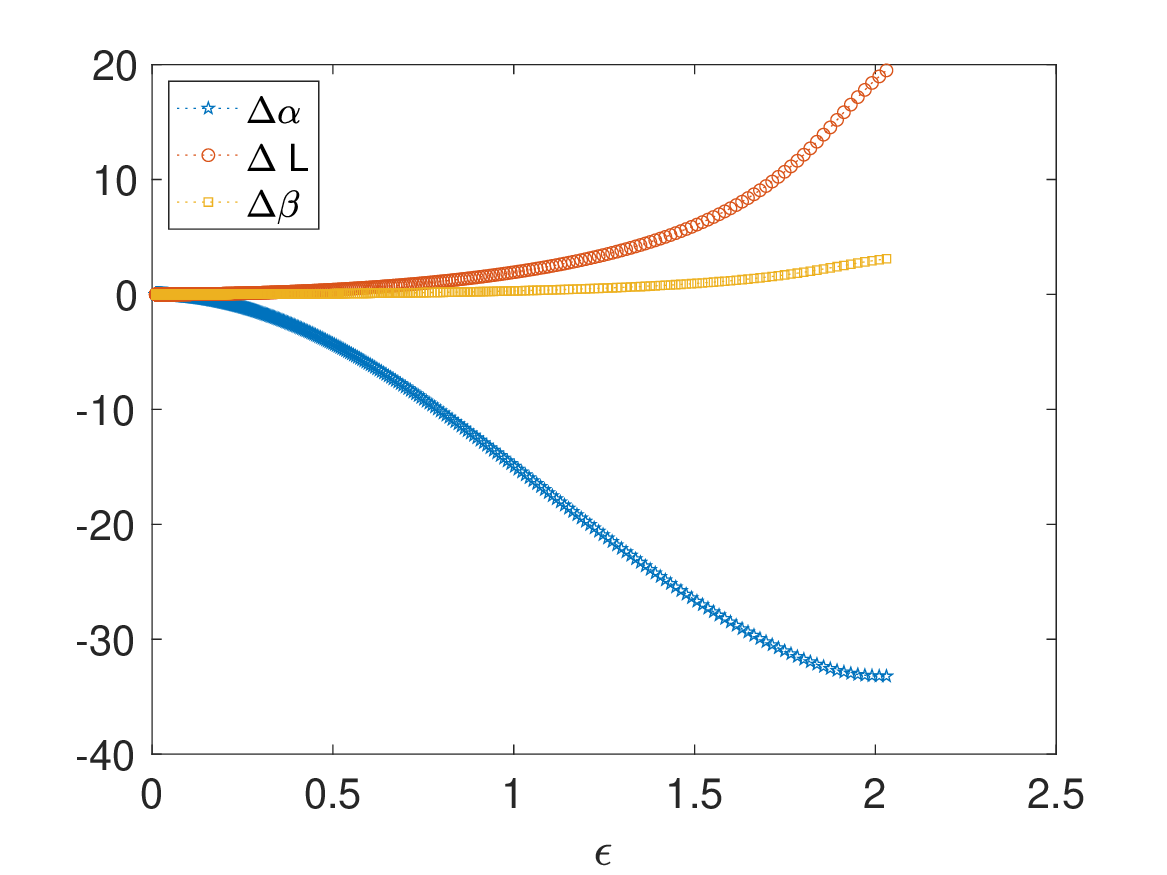}\includegraphics[width=0.3\textwidth]{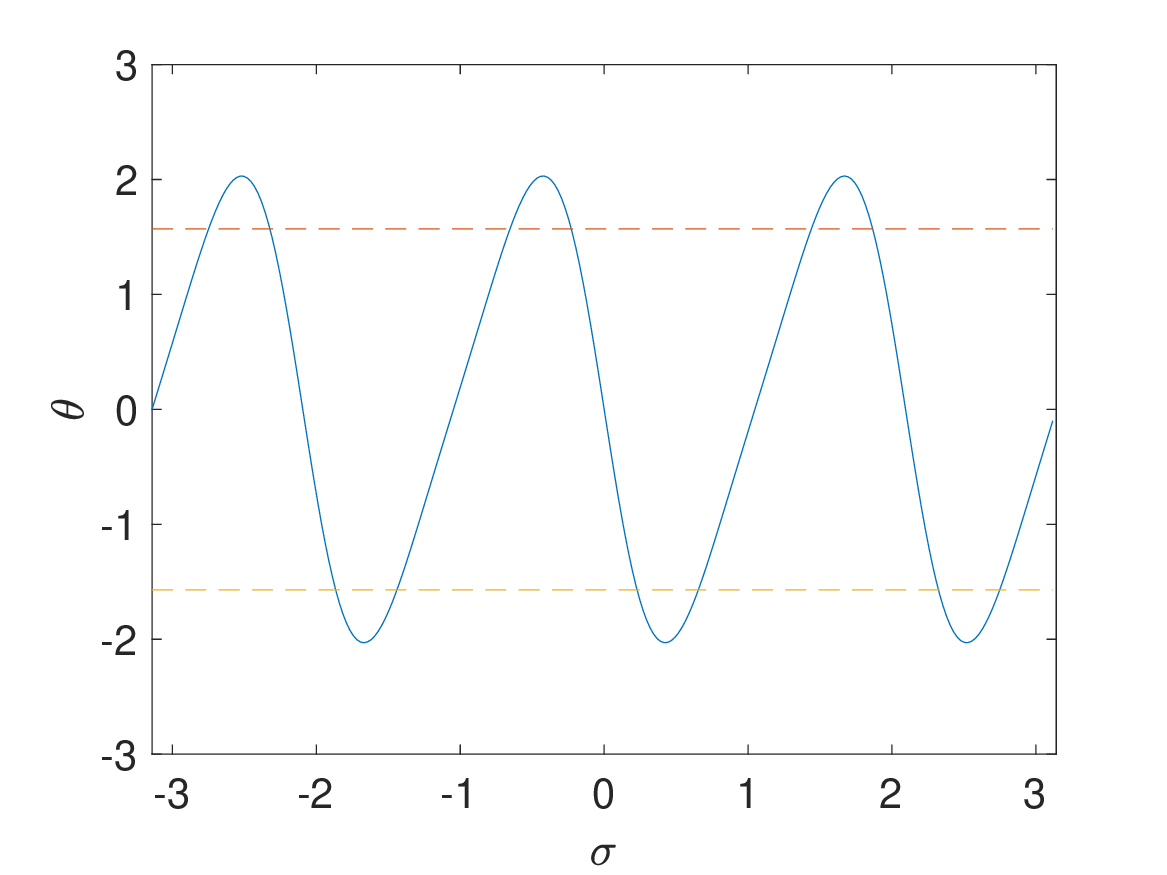}\includegraphics[width=0.3\textwidth]{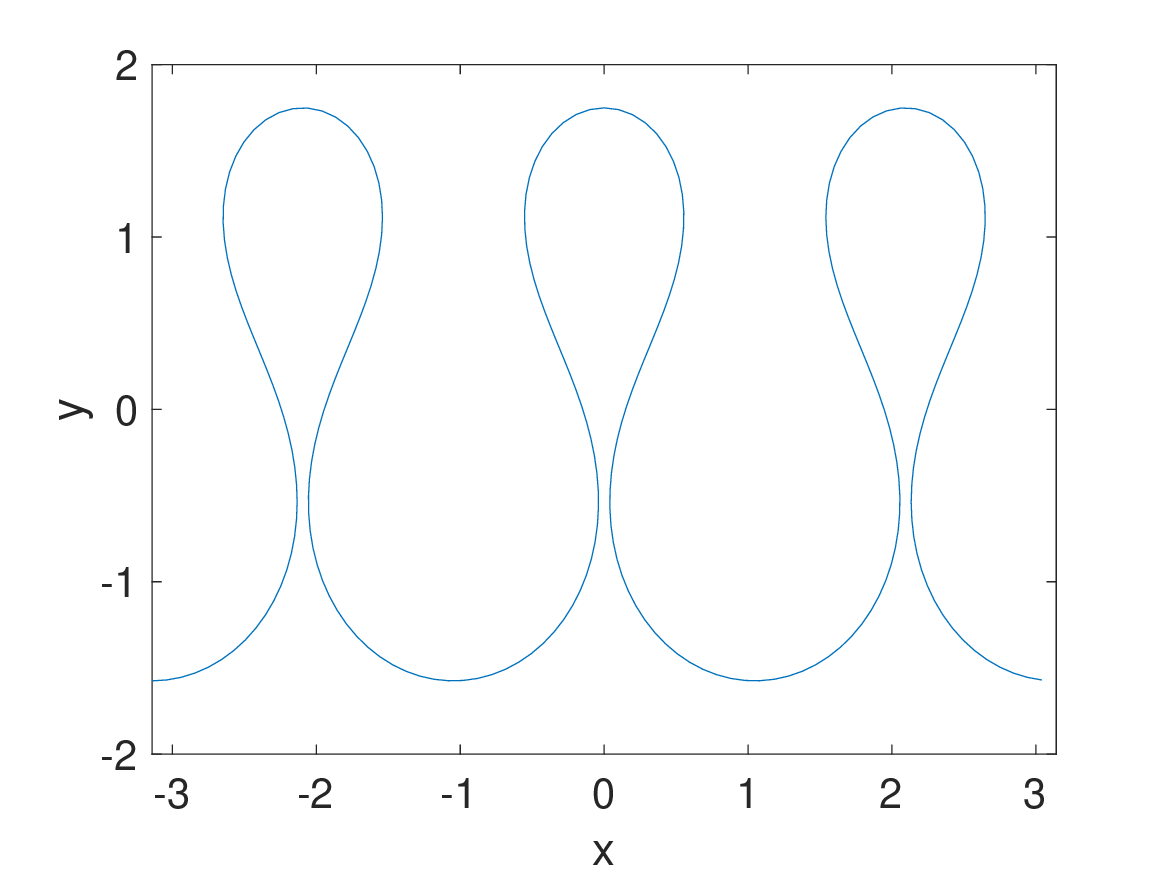}}
\caption{\it \textbf{Left:} The change in parameters $L,\alpha$ and $\beta$ over a computed branch of solutions to \eqref{StandingLinModel}, starting with $\alpha=37, \beta=1,L=2\pi$. \textbf{Center:} The standing wave's tangent angle $\theta$ on the largest wave on this part of the branch. The dotted lines mark $|\theta|=\pi/2$. \textbf{Right:} The standing wave's interface profile of the largest wave on this part of the branch.  The branch limits on a self intersecting profile. \label{LinUBranch3}  }
\end{figure}

\subsubsection{Linear Velocity Model}
For model \eqref{LinModel}, the vertically traveling solutions ($\beta\ne0$) solve
\begin{equation}\label{StandingLinModel}
 \beta \cos(\theta)=1+(\alpha-1)\frac{2\pi}{L}\theta_{\sigma}+4\frac{(2\pi)^3}{L^3}\theta_{\sigma\sigma\sigma} 
  \end{equation}
  where $\alpha$ is a physical constant, $\beta$ is the wave's speed and $L$ is the length of the interface as given in \eqref{lengthFormula} above.

The quasi-Newton iteration used to compute solutions to \eqref{StandingLinModel} requires an initial guess.  We look for small amplitude solutions by expanding in $\epsilon$ (as in \cite{akers2014spectrum}, \cite{akers2012wilton}, \cite{creedon2021high},  \cite{nicholls2007boundary}, \cite{nicholls2005analyticity}), 
\begin{subequations}\label{PertExp}
\begin{eqnarray}
\theta&=&\epsilon\theta_1+\epsilon^2\theta_2+O(\epsilon^3),\\
\beta&=&\beta_0+\epsilon\beta_1+\epsilon\beta_2+O(\epsilon^3),\\
\alpha&=&\alpha_0+\epsilon\alpha_1+\epsilon\alpha_2+O(\epsilon^3).
\end{eqnarray}
\end{subequations}
The leading order solution is
\[ \theta_1=A\cos\left(\frac{L\sqrt{\alpha_0-1}}{4\pi}\sigma\right)+B\sin\left(\frac{L\sqrt{\alpha_0-1}}{4\pi}\sigma\right), \]
coupled with
and $L=2\pi+O(\epsilon^2)$ and $\beta_0=1,\beta_1=0$.  Enforcing periodicity for $\sigma\in [0,2\pi)$ gives $\alpha_0=(\frac{4\pi k_{0}}{L})^2+1$ for $k_{0}\in\mathbb{N}$.   At small amplitude $L=2\pi$, so $\alpha_0=(2k_{0})^2+1$ and the lowest non-trivial frequency is $\alpha_0=5$.  Also, we search for waves which are odd in $\sigma$, setting $A=0$ and $B=1$ (the latter can be thought of as a definition of $\epsilon$).   This leading order solution is enough to initialize a quasi-Newton iteration and compute traveling waves.  To justify our choice to expand the parameters $\alpha$ and $\beta$, rather than fixing them along solution branches, we develop the higher order terms in the series in the example case, $k_{0}=1$.\\

The $O(\epsilon^2)$ problem is 
\begin{multline}\nonumber
(\alpha_0-1)\theta_{2,\sigma}+4\theta_{2,\sigma\sigma\sigma} 
\\
=
\beta_2-\beta_0\theta_1^2/2=\beta_2-\frac 12 \sin^2(\sigma)=\beta_2-\frac 14+\frac 14\cos(2\sigma)-\alpha_1\theta_{1,\sigma}.
\end{multline}
This equation is solvable, with solution
\[ \theta_2=-\frac{1}{96}\sin(2\sigma),\qquad \alpha_1=0, \qquad\text{and}\qquad \beta_2=\frac 14.\]

The $O(\epsilon^n)$ problem is 
\begin{equation}\label{OnCorrection}
(\alpha_0-1)\theta_{n,\sigma}+4\theta_{n,\sigma\sigma\sigma}= \beta_n-\alpha_{n-1}\theta_{1,\sigma}+G_n(\theta_{n-1},\alpha_{n-2},\beta_{n-1},...).
\end{equation}
The functions $G_n$ depend on terms which have been determined at previous orders.  The functional dependence of $G_n$ is even in $\sigma$ and orthogonal to $\sin(\sigma)$.  The parameter $\beta_n$ can determined by enforcing that the right hand side of \eqref{OnCorrection} has zero mean.  The parameter $\alpha_{n-1}$ is found by forcing the right hand to be orthogonal to $\cos(\sigma)$.  That the expansions for $\alpha$ and $\beta$ include non-zero terms requires that $\alpha$ and $\beta$ are non-constant in amplitude, and thus can be thought of as part the solution (rather than parameters to be specified).  This is similar to the velocity expansions used to compute Stokes' waves in the water wave problem \cite{craik2005george}.  The expansion in $\alpha$ and $\beta$ is effectively a Lyapunov-Schmidt reduction of this problem, as in 
 \cite{Akers_2025}, \cite{ehrnstrom2019bifurcation}.  
 Branches of traveling waves were computed in both the linear and nonlinear  models, for $k_{0}=1,2$, and $3$.  Each branch of waves bifurcates from an infinitesimal sinusoid at a single value of 
$\alpha$. Global branches of traveling waves are computed, connected to the small amplitude wave.

The results for the linear velocity closure, \eqref{StandingLinModel}, are depicted in Figures \ref{LinUBranch}, \ref{LinUBranch2}, and \ref{LinUBranch3}. 
The left panels depict the changes in the speed $\Delta\beta=\beta-1$, changes in the parameter $\Delta\alpha=\alpha-\alpha_0$ and changes in the length of the curve $\Delta L=L-2\pi$.  The center panel depicts the tangent angle profile, $\theta$, of the extreme wave on the branch.   Dashed horizontal lines mark the angle $\pi/2$; wave profiles with tangent angle above this threshold are overhanging and require coordinate-free models (such as the present models) 
for their computation. Weakly nonlinear models such as the Kuramoto-Sivashinsky
equation are limited only to admit non-overhanging waves, and are therefore insufficient to capture this phenomenon.  The right panel depicts the interface corresponding to the extreme wave.  For model \eqref{StandingLinModel} these waves are approximately self-intersecting.

The method used to compute solutions to \eqref{StandingLinModel} is in terms of the tangent angle $\theta$ and does not require that the wave profile be reconstructed.  We do not observe any singularity in the tangent angle profiles. On the contrary the continuation method is capable of computing waves with larger tangent angle profiles.  Beyond the amplitude ranges reported in Figures \ref{LinUBranch}, \ref{LinUBranch2}, and \ref{LinUBranch3} the tangent angles reconstruct into self intersecting interfaces.  Numerically we define a profile as being near self-intersection if two non-adjacent points on the curve are closer than then the 90\% of the pseudo-arclength grid spacing.  This criterion was used to terminate the computations of branches of waves depicted in Figures \ref{LinUBranch}, \ref{LinUBranch2}, and \ref{LinUBranch3} just before the wave profiles self-intersect.

\begin{figure}[tp]
\centerline{\includegraphics[width=0.5\textwidth]{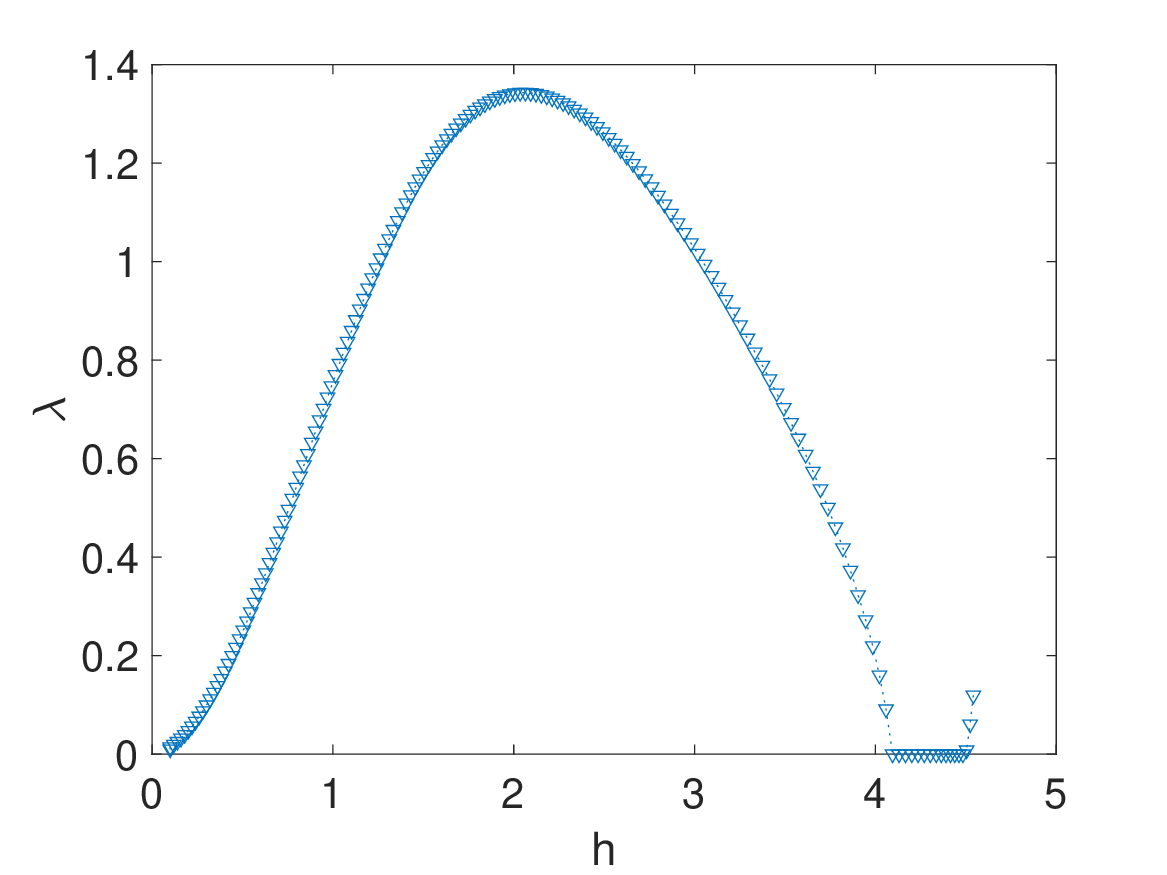}\includegraphics[width=.5\textwidth]{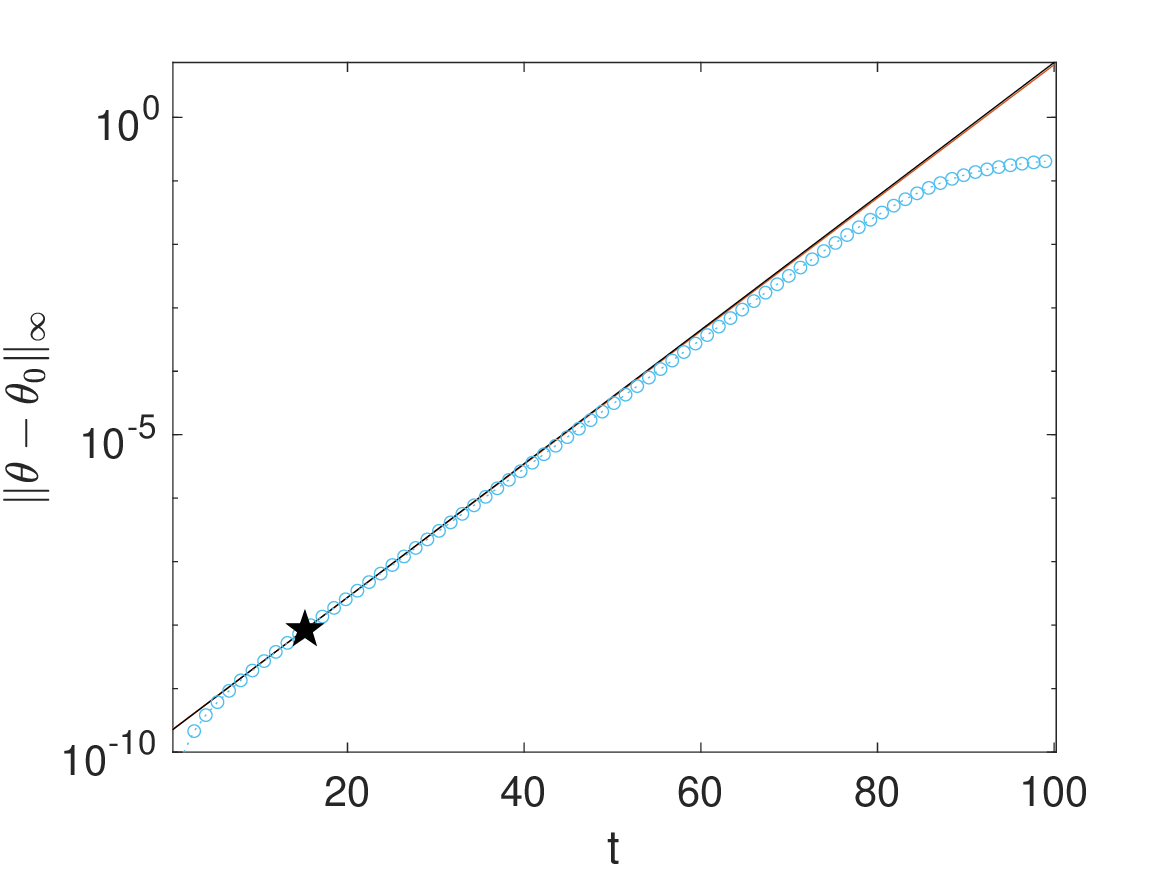}}
\caption{\it \textbf{Left:} Estimates of the eigenvalues of the branch of waves in Figure \ref{LinUBranch}.  The flat state in this configuration is superharmonically stable (but subharmonically unstable). \textbf{Right:} The method used to estimate eigenvalues fits the change in $\theta$ with $y=C\exp(\lambda t)$.  This panel depicts the fit to wave of total displacement $h=0.5044$ from the left panel.  The star marks the location where the growth rate is estimated. \label{LambdaPlot1} }
\end{figure}

\begin{figure}[tp]
\centerline{\includegraphics[width=0.5\textwidth]{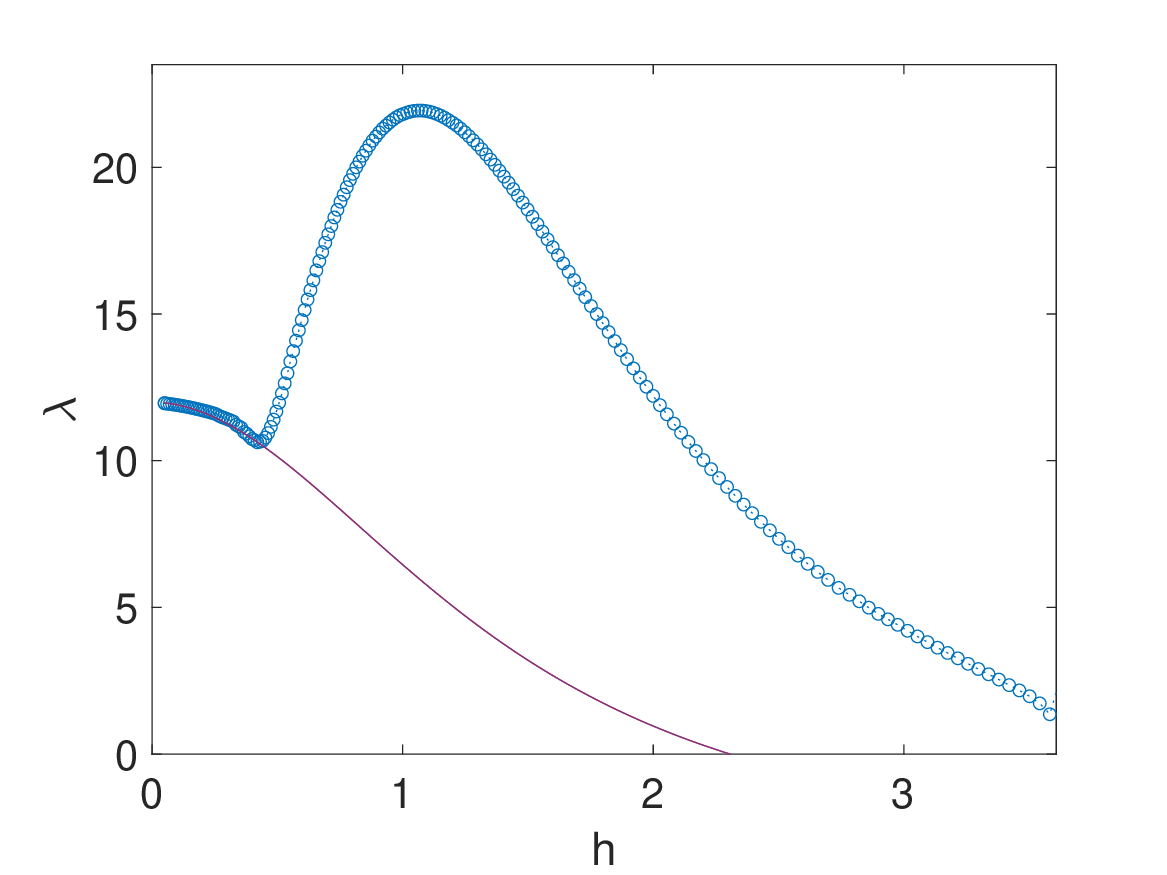}\includegraphics[width=.5\textwidth]{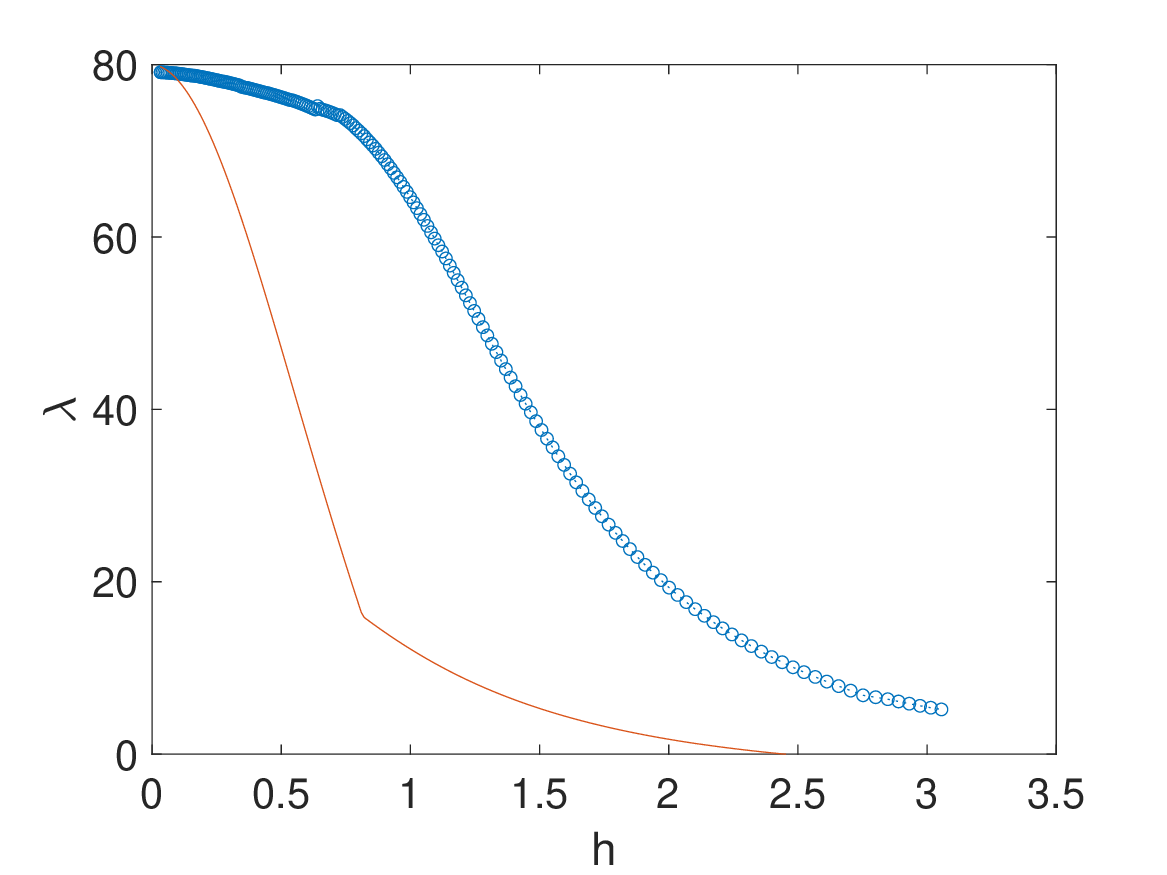}}
\caption{\it \textbf{Left:} Estimates of the eigenvalues of the branch of waves in Figure \ref{LinUBranch2}. \textbf{Right:} Estimates of the eigenvalues of the branch of waves in Figure \ref{LinUBranch3}.  These two branches bifurcate from a flat state configuration which is linearly unstable to supercritical perturbations.  The parameter $\alpha$ varies along the branch.  The maximally unstable supercritical eigenvalue for the flat state at each value of $\alpha$ are plotted for comparison. \label{LambdaPlot2} }
\end{figure}

The wave profiles computed from \eqref{StandingLinModel} exist at values of $\alpha$ for which the time dynamic model is well-posed.  As such one can investigate the stability of these waves.  As a simple test of stability, we evolve these waves in their parent equation using Fourier collocation in space and an IMEX time stepper, a method described in detail in \cite{ANZIAM}.  The linear problem has solutions $\theta(\sigma,t)=\exp(ik\sigma+\lambda t)$ with
\[ \lambda(k)=(-4k^4+(\alpha-1) k^2)\]
which means there are unstable linear modes in the set $0<|k|< 1/2\sqrt{\alpha-1} $.  For $\alpha_0=5$ there are no unstable modes $k\in\mathbb{Z}$, so the flat state problem is superharmonically stable (but subharmonically unstable).  For $\alpha_0=17$ there is one superharmonically unstable eigenfunction (with wavenumber $k=1)$.  For $\alpha_0=35$ there are two superharmonically unstable eigenfunctions (with wavenumbers $k=1$ and $k=2$).   To observe the manifestation of these instabilities we evolve the wave profiles in the time dynamic problem with perturbed initial data
\[ \theta(\alpha,0)=\bar{\theta}(\alpha)+\delta \sin(\alpha)+\delta\sin(2\alpha),\]
with $\delta =10^{-8}$.  We then fit the spatial infinity norm of the difference between the evolution and the initial data to an exponential as an estimate of the growth rate.  This fitting process is depicted in the right panel of Figure \ref{LambdaPlot1}.  The growth rate was estimated at the time when the profile changed by two orders of magnitude. 

Estimates of the most unstable eigenvalue of each wave on the branch in model \eqref{StandingLinModel} are in Figure \ref{LambdaPlot1} and Figure \ref{LambdaPlot2}.  For the $k_{0}=2$ and $k_{0}=3$ branches all of the waves are observed to be unstable.  The small amplitude instability is predicted correctly by the maximally unstable linear solution.  As amplitude increases we observe a transition; the observed instability seems to change which eigenfunction is most unstable.  Since $\alpha$ is changing along the branch, we can compare the observed instability to linear growth rates at for each value of $\alpha$ to the numerically estimated eigenvalues.  These are compared in Figure \ref{LambdaPlot2}.

For the branch of waves computed beginning with $k_{0}=1$, the linear problem is superharmonically stable for all $\alpha$ on the branch.  The time evolution of the nonlinear problem exhibits instability at small amplitude, with eigenvalues that grow with amplitude.  For a narrow range of amplitudes, $4.095\le h\le 4.4836$ we do not observe instability.  It is possible that these waves are also unstable.  For example the instability may be associated with a more exotic eigenfunction or may have an eigenvalue with real part less than $10^{-3}$.  
In future work, a more detailed analysis, such as that in \cite{deconinck2006computing}, \cite{deconinck2011instability}, \cite{trichtchenko2016instability}, could be used to track down instabilities in this case.

\subsubsection{Nonlinear Velocity Model}

\begin{figure}[tp]
\centerline{\includegraphics[width=0.3\textwidth]{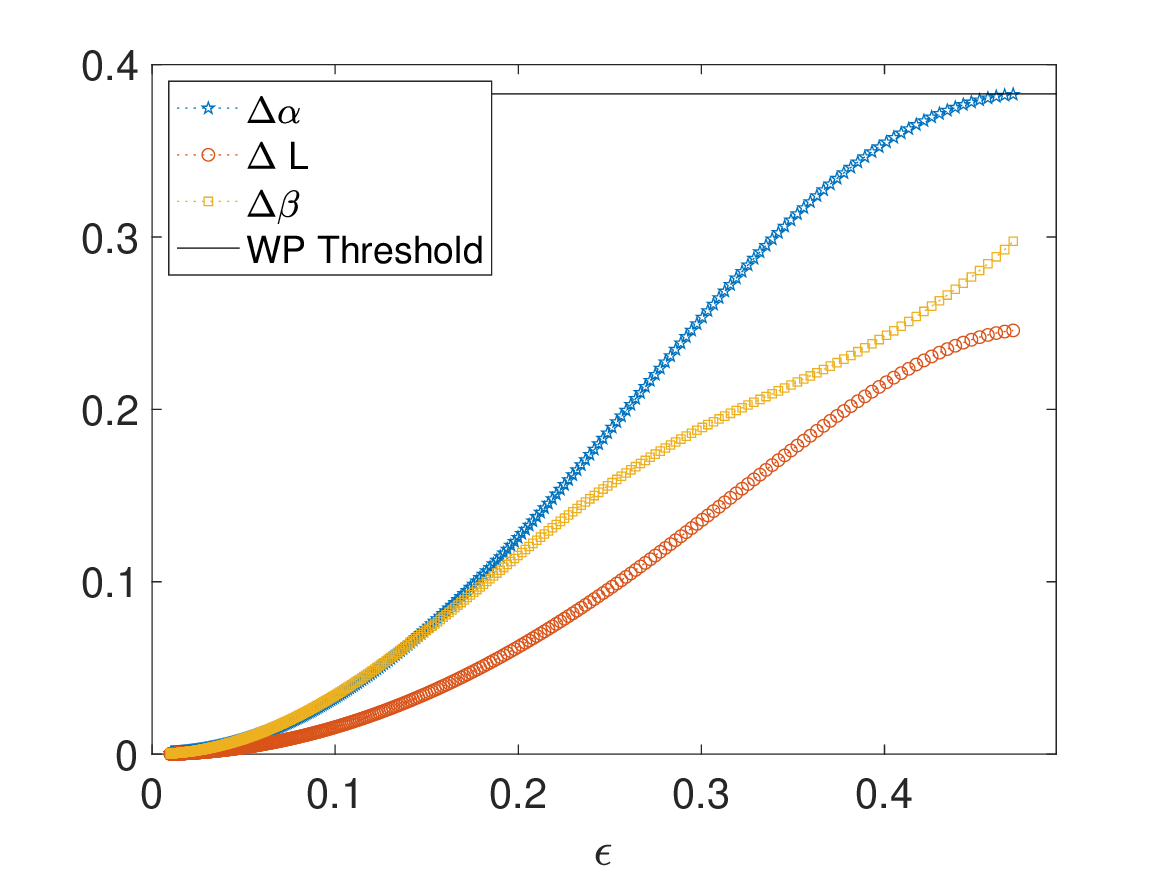}\includegraphics[width=0.3\textwidth]{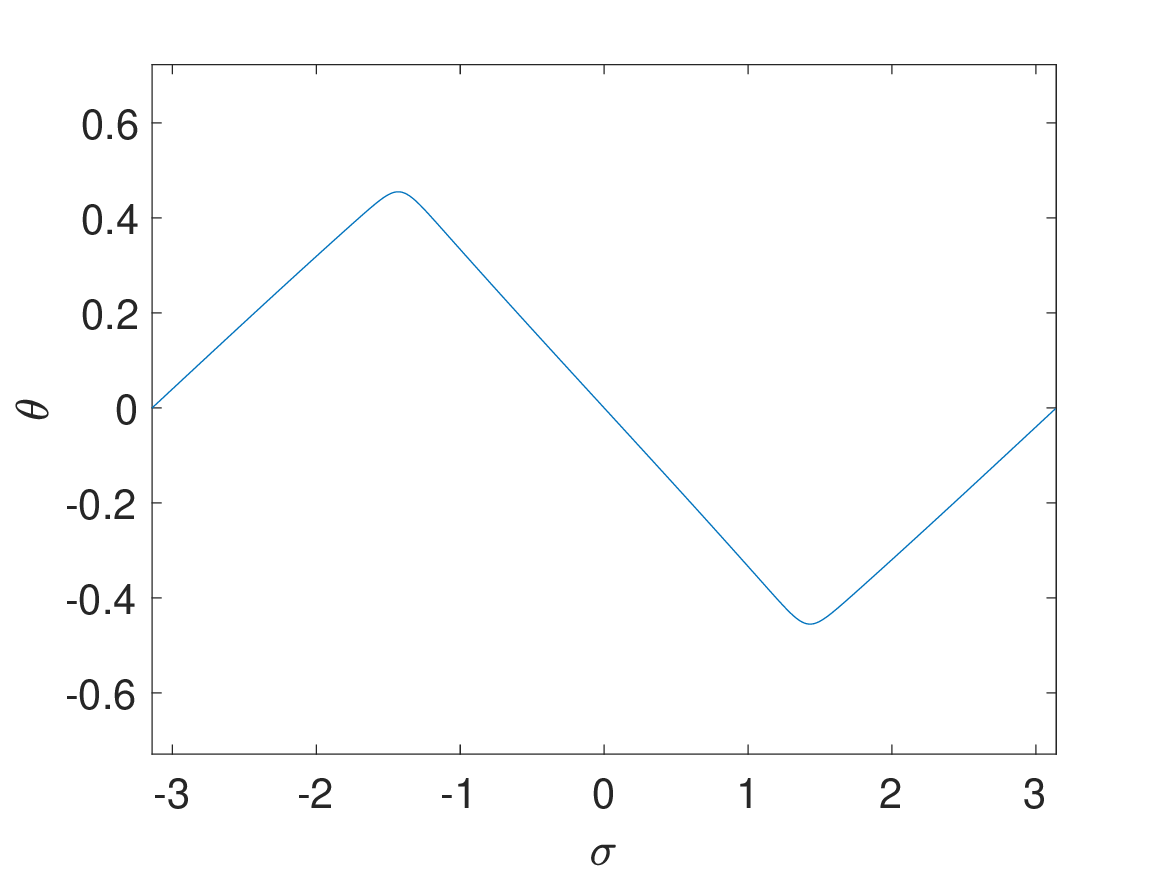}\includegraphics[width=0.3\textwidth]{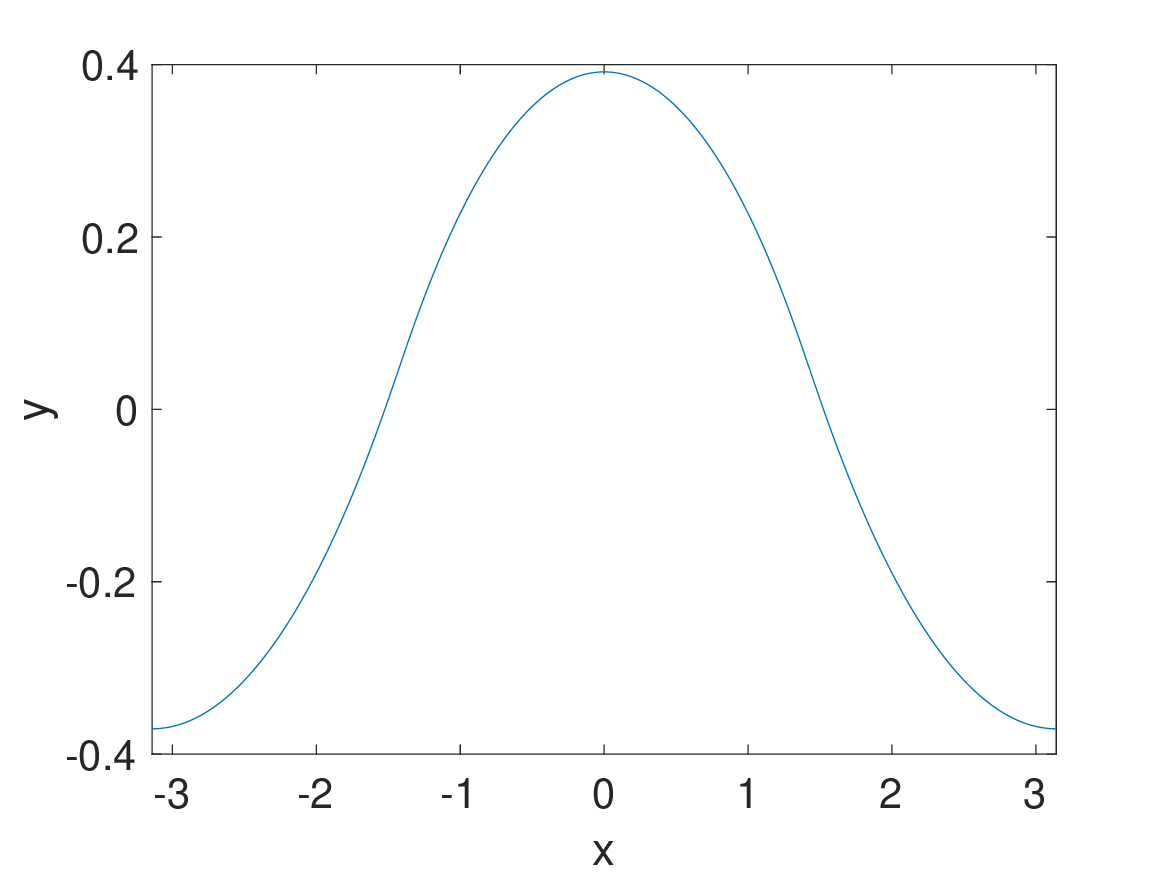}}
\caption{\it \textbf{Left:} The change in parameters $L,\alpha$ and $\beta$ over a computed branch of solutions to \eqref{StandingNonLinModel}, starting with $\alpha\approx-3.383, \beta=1,L=2\pi,n=1$. \textbf{Center:} The standing wave's tangent angle $\theta$ on the largest wave on this part of the branch.  \textbf{Right:} The standing wave's interface profile of the largest wave on this part of the branch. The branch seems to limit on a wave with a curvature singularity. \label{NonLinUBranch1}}
\end{figure}

\begin{figure}[tp]
\centerline{\includegraphics[width=0.3\textwidth]{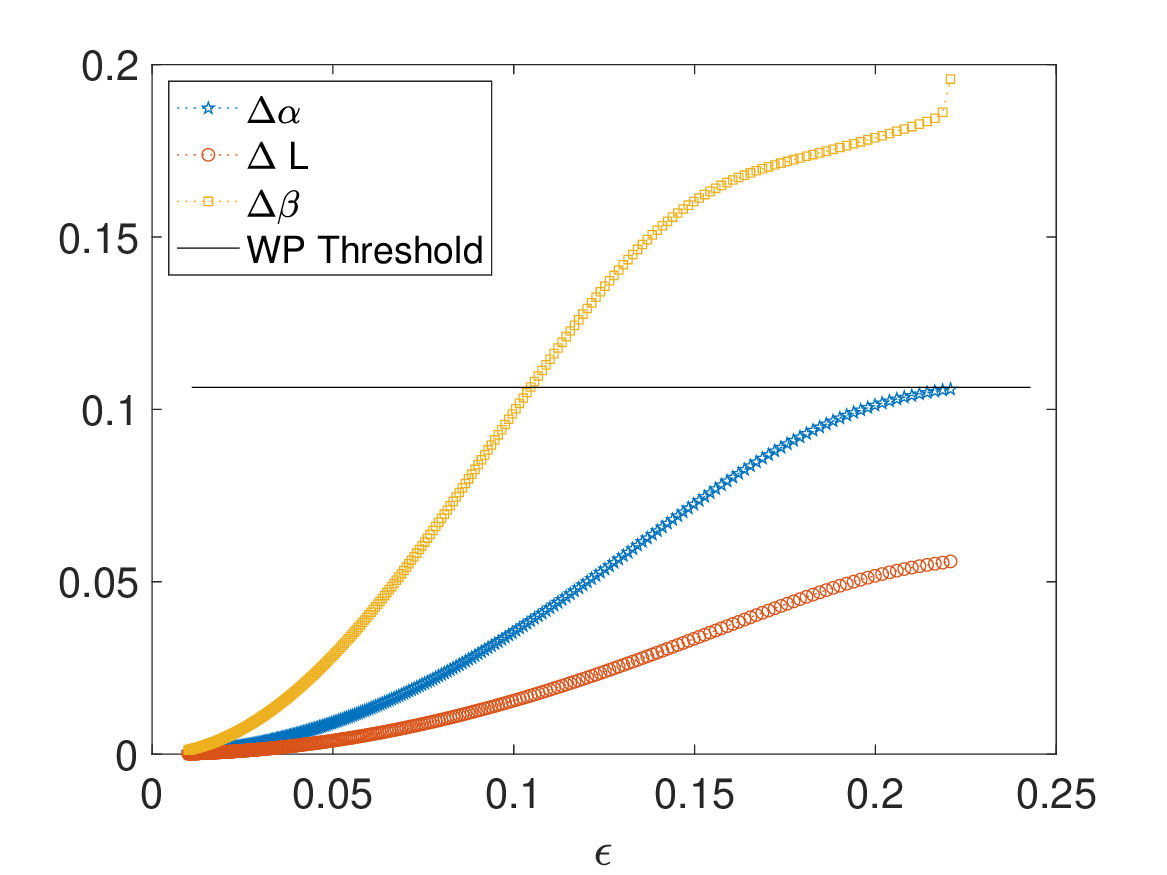}\includegraphics[width=0.3\textwidth]{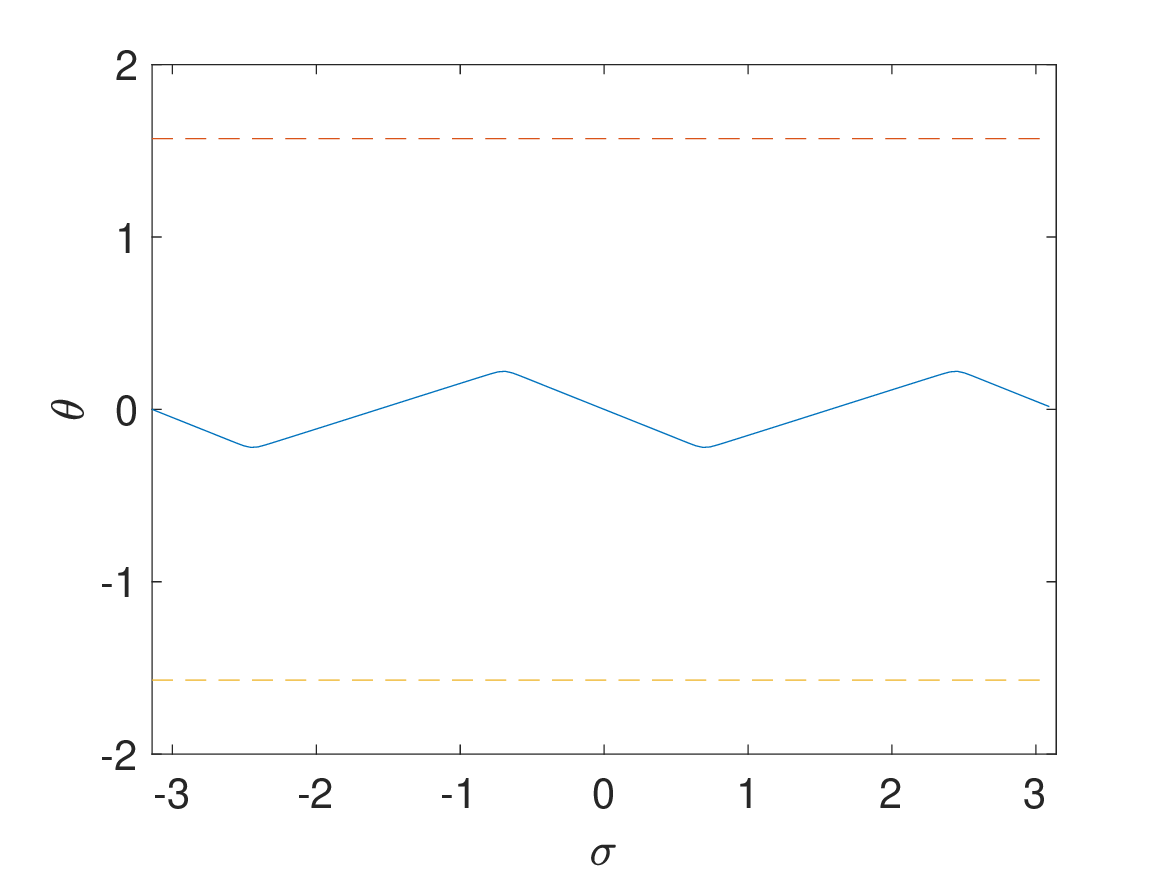}\includegraphics[width=0.3\textwidth]{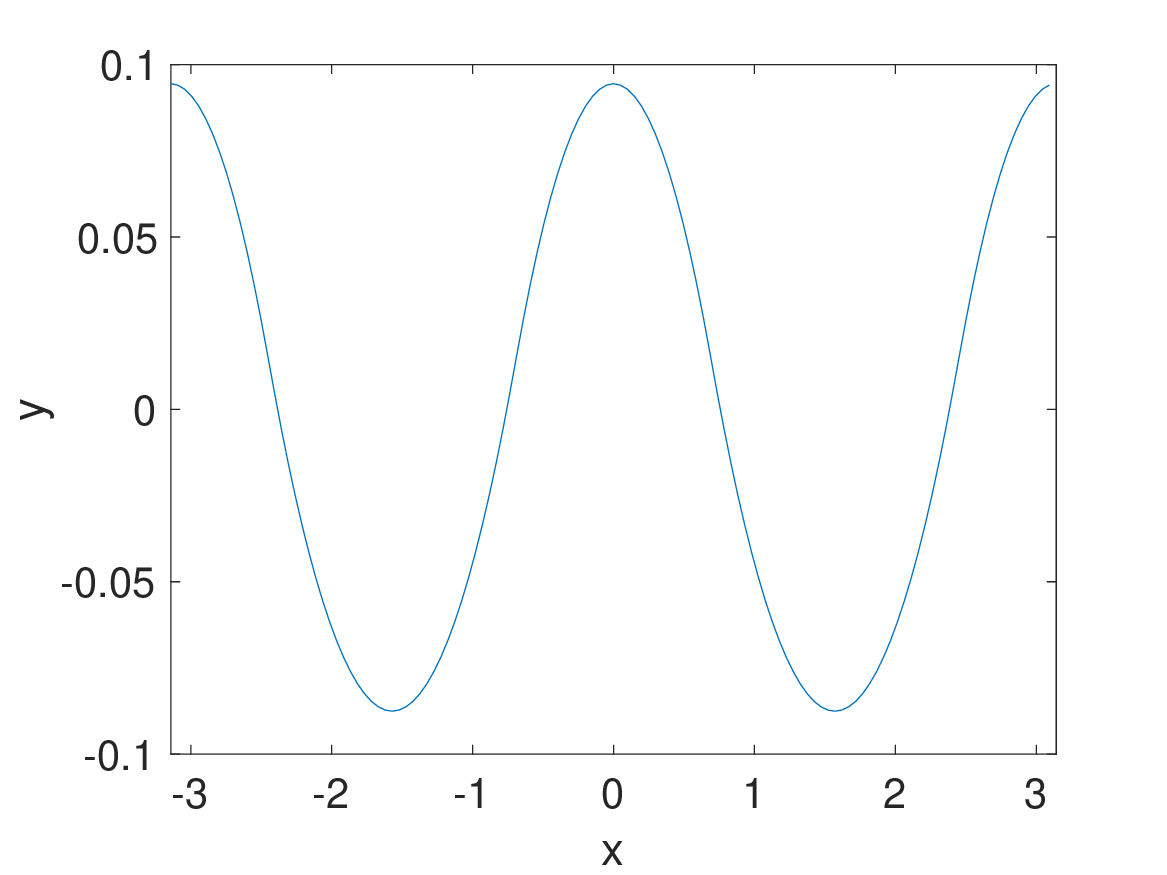}}
\caption{\it \textbf{Left:} The change in parameters $L,\alpha$ and $\beta$ over a computed branch of solutions to \eqref{StandingNonLinModel}, starting with $\alpha\approx-3.1064, \beta=1,L=2\pi, n=2$. \textbf{Center:} The standing wave's tangent angle $\theta$ on the largest wave on this part of the branch. \textbf{Right:} The standing wave's interface profile of the largest wave on this part of the branch. The branch seems to limit on a wave with a curvature singularity. \label{NonLinUBranch2}}
\end{figure}

\begin{figure}[htp]
\centerline{\includegraphics[width=0.3\textwidth]{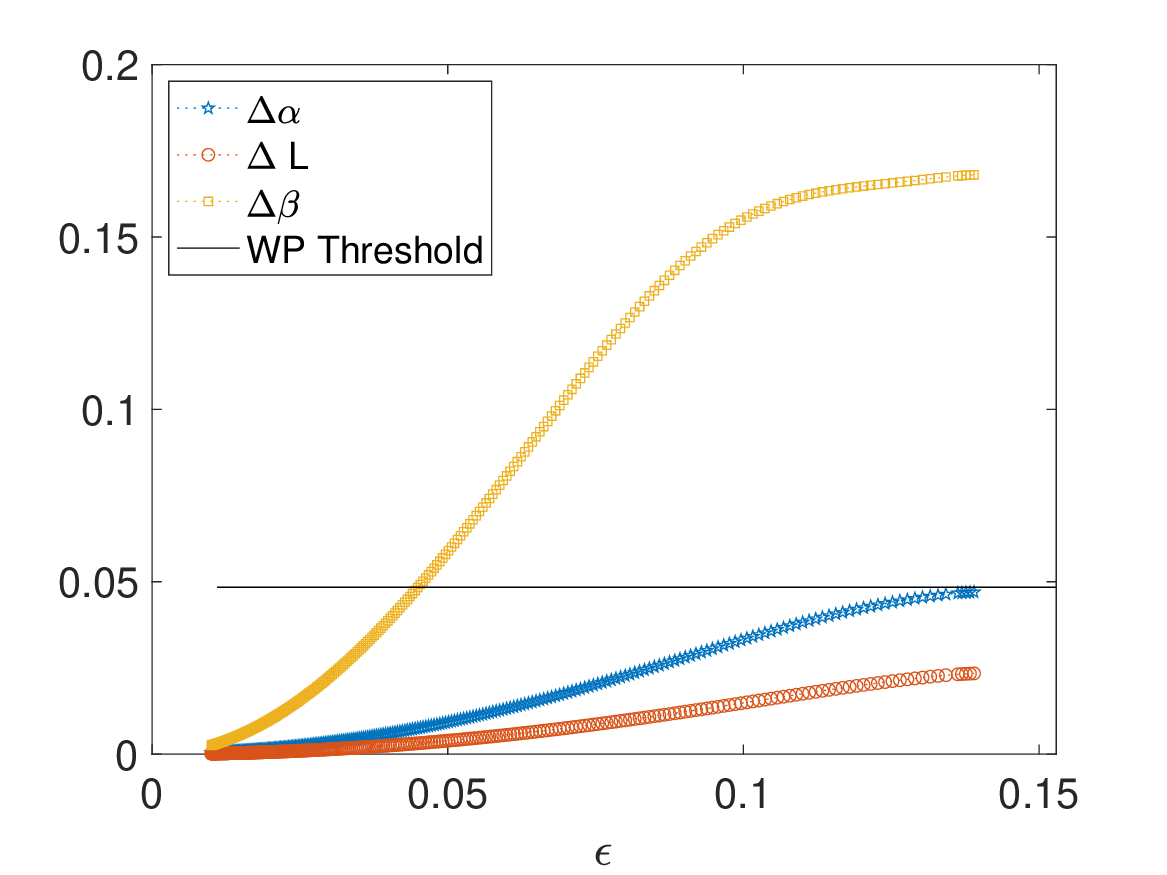}\includegraphics[width=0.3\textwidth]{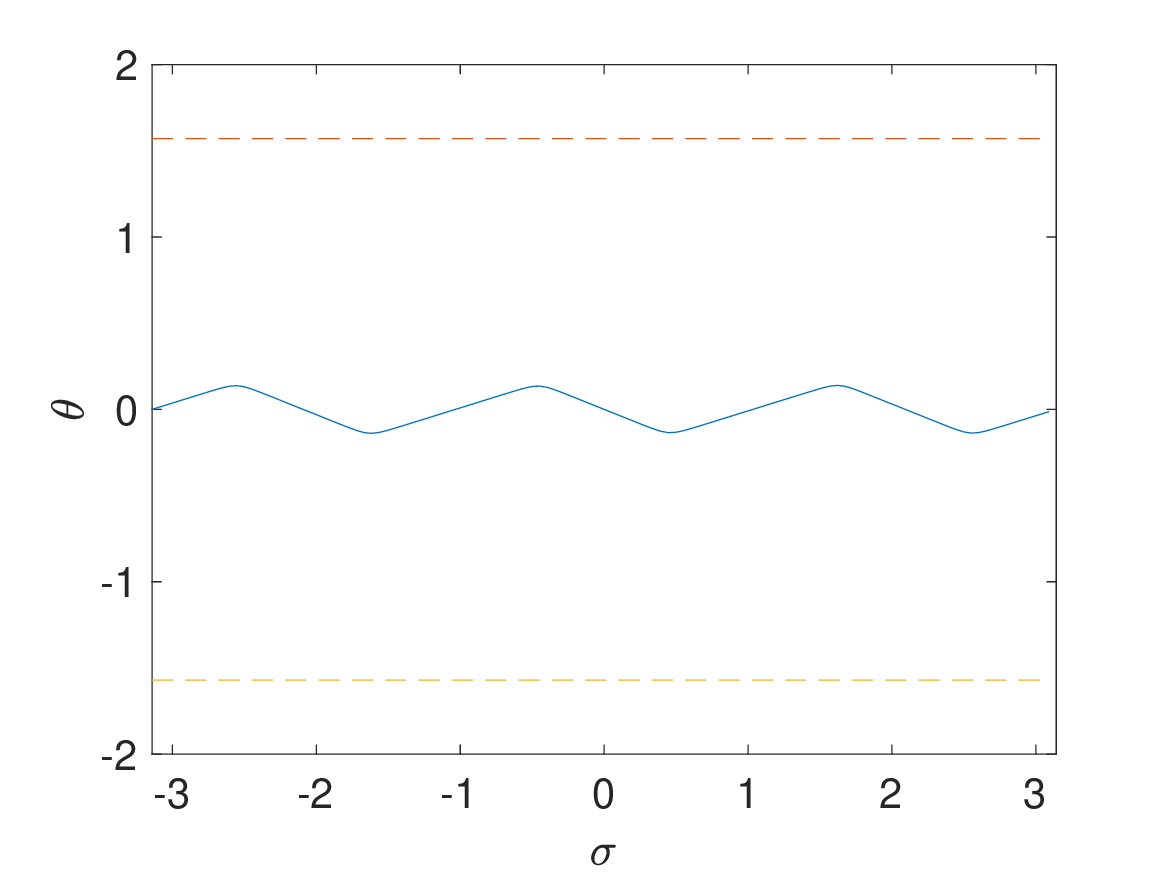}\includegraphics[width=0.3\textwidth]{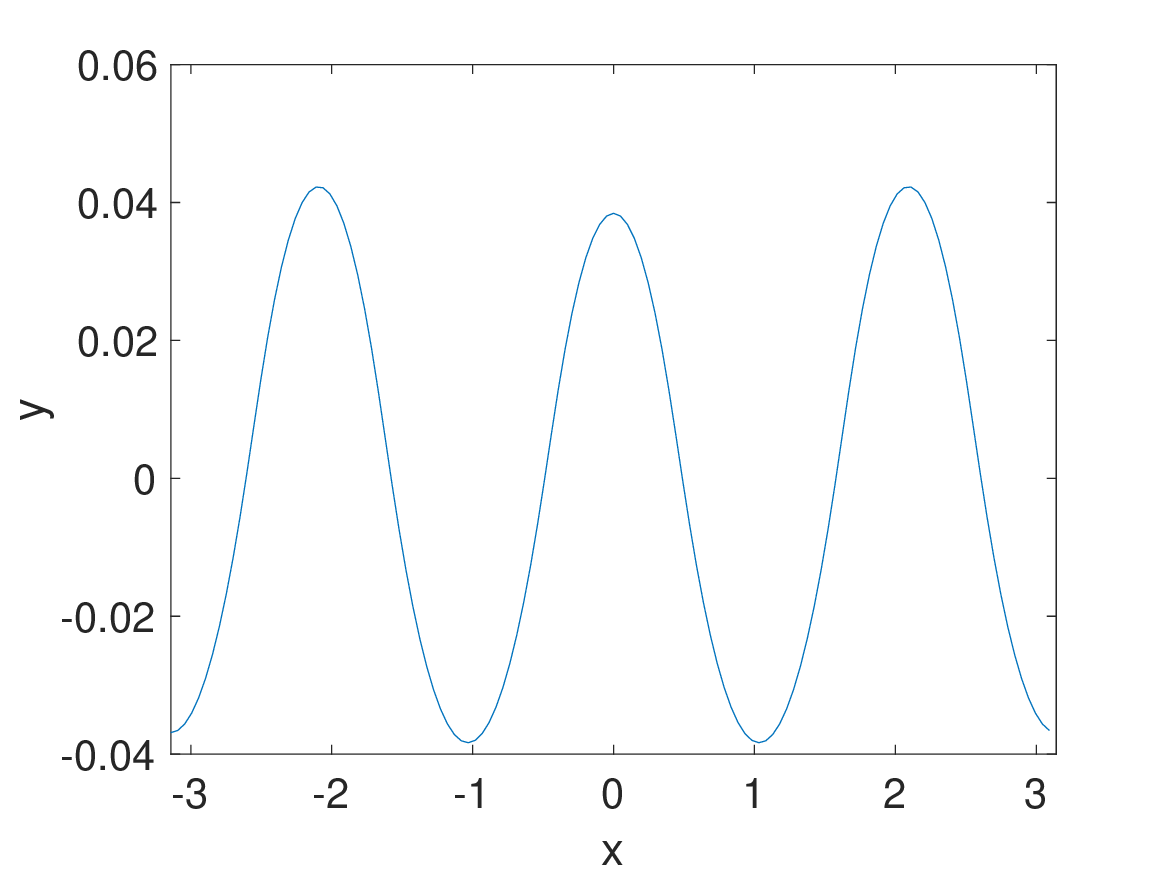}}
\caption{\it \textbf{Left:} The change in parameters $L,\alpha$ and $\beta$ over a computed branch of solutions to \eqref{StandingNonLinModel}, starting with $\alpha\approx-3.0484, \beta=1,L=2\pi$,n=3. \textbf{Center:} The standing wave's tangent angle $\theta$ on the largest wave on this part of the branch. The dotted lines mark $|\theta|=\pi/2$. \textbf{Right:} The standing wave's interface profile of the largest wave on this part of the branch. The branch seems to limit on a wave with a curvature singularity.\label{NonLinUBranch3} }
\end{figure}
For model \eqref{NonLinModel}, the forward propagating solutions $(\beta\ne 0)$ can be written
\begin{multline}\label{StandingNonLinModel}
  \beta\cos(\theta)=
  \\
  1+(\alpha-1)\frac{2\pi}{L}\theta_{\sigma}+\alpha^2(\alpha+3)\frac{(2\pi)^3}{L^3}\theta_{\sigma\sigma\sigma}+\left(1+\frac{\alpha}{2}\right)\kappa^2+\left( 2\alpha+5\alpha^2-\frac{\alpha^3}{3} \right)\kappa^3.
  \end{multline}
 Expanding this equation using the ansatz in \eqref{PertExp} gives a leading order problem
 \[(\alpha_0-1)\theta_{1,\sigma}+\alpha_0^2(\alpha_0+3)\theta_{1,\sigma\sigma\sigma}=0, \]
 which has periodic solutions
 \[ \theta_1=A\cos\left(\frac{\sqrt{\alpha_0-1}}{|\alpha_0|\sqrt{\alpha_0+3}}\sigma\right)+B\sin\left(\frac{\sqrt{\alpha_0-1}}{|\alpha_0|\sqrt{\alpha_0+3}}\sigma\right), \]
 with $\beta_0=1$.  Brief inspection of \eqref{StandingNonLinModel} reveals that it could support odd $\theta$, but cannot have even $\theta$; we set $A=0$ and absorb $B$ into the definition of $\epsilon$.  The linearization tells us that periodic waves for $\theta\in(0,2\pi)$ may exist only when 
 \[ \alpha_0^2(\alpha_0+3)k_{0}^2-(\alpha_0-1)=0, \qquad k_{0}\in\mathbb{N}.\]
 For $k_{0}=1$ the single real root of this cubic is $\alpha_0\approx -3.383$.

Branches of traveling waves were also computed in the fully nonlinear coordinate-free model \eqref{StandingNonLinModel}.  These results are depicted in Figures \ref{NonLinUBranch1}, \ref{NonLinUBranch2}, and \ref{NonLinUBranch3}.   The left panels depict the changes in the speed $\Delta\beta=\beta-1$, changes in the parameter $\Delta\alpha=\alpha-\alpha_0$ and changes in the length of the curve $\Delta L=L-2\pi$. The value of $\alpha_0$ from which these waves bifurcate is in the regime for which the time dynamic model (as discussed in the introduction) is linearly ill-posed.  A horizontal line marks the value of $\alpha$ for which the model becomes well posed.  We observe that the branches of waves do not cross this threshold.  As $\alpha$ approaches the well-posedness threshold, the wave profile becomes singular.  The tangent angle develops a discontinuous derivative, corresponding to a singularity in curvature.   The center panel depicts the tangent angle profile, $\theta$, of the extreme wave on the branch.  The right panel depicts the interface corresponding to the extreme wave.

\section*{Acknowledgments}
The second author acknowledges support from the Air Force Office of Scientific Research and the Joint Directed Energy Transition Office.  
The third author is grateful to the National Science Foundation for support through grant DMS-2307638.

{\bf Disclaimer:} This report was prepared
as an account of work sponsored by an agency of the United States
Government. Neither the United States Government nor any agency
thereof, nor any of their employees, make any warranty, express or
implied, or assumes any legal liability or responsibility for the
accuracy, completeness, or usefulness of any information,
apparatus, product, or process disclosed, or represents that its
use would not infringe privately owned rights. Reference herein to
any specific commercial product, process, or service by trade name,
trademark, manufacturer, or otherwise does not necessarily
constitute or imply its endorsement, recommendation, or favoring
by the United States Government or any agency thereof. The views
and opinions of authors expressed herein do not necessarily state
or reflect those of the United States Government or any agency
thereof.

\bibliographystyle{plain}
\bibliography{WiltonSheets}{}

\end{document}